\documentclass[11pt]{article}

\usepackage{enumitem}
\usepackage{amsmath, amssymb, amsthm, dsfont}
\numberwithin{equation}{section}
\usepackage{stmaryrd}
\usepackage{geometry}
\usepackage{graphicx}
\usepackage{float}
\usepackage{wrapfig}
\usepackage{subcaption}
\usepackage{comment}
\usepackage[table]{xcolor}
\usepackage{multirow}
\usepackage{tikz}
\usetikzlibrary{arrows.meta, positioning, decorations.pathreplacing, calc, decorations.markings, tikzmark,calligraphy,shadows.blur,fit,backgrounds}
\usepackage{hyperref}
\usepackage{cleveref}
\crefname{enumi}{}{} 

\usepackage{authblk}
\usepackage[toc,page]{appendix}

\usepackage{xr-hyper}
\usepackage[backend=biber, url=false, doi=true, isbn=false, sorting=none, citestyle=numeric-comp, maxbibnames=10]{biblatex}
\DeclareNameAlias{author}{family-given}
\DeclareFieldFormat[article]{journaltitle}{\mkbibemph{#1}} 
\DeclareFieldFormat[book]{title}{\mkbibemph{#1}} 
\DeclareFieldFormat{doi}{%
  (doi\addcolon\space
  \href{https://doi.org/#1}{#1})} 
\DeclareFieldFormat[article]{volume}{#1} 
\DeclareFieldFormat[article]{pages}{#1} 
\DeclareFieldFormat[article]{title}{#1\addperiod} 
\renewbibmacro*{in:}{} 

\def\Aa{\emph{A. albopictus}}

\definecolor{redLight}{HTML}{F7EDED}
\definecolor{purpleLight}{HTML}{EFEDF7}
\definecolor{redDark}{HTML}{3D093B}
\definecolor{redDLight}{HTML}{F6C1F4}
\definecolor{blueDark}{HTML}{0B093D}
\definecolor{blueDLight}{HTML}{C3C1F6}
\definecolor{purpleDark}{HTML}{213D09}
\definecolor{purpleDLight}{HTML}{DAF6C1}
\definecolor{aquamarine}{rgb}{0.13, 0.68, 0.8} 
\definecolor{dkorange}{RGB}{214, 141, 51} 
\definecolor{maroon}{RGB}{128, 0, 0}

\def\Lil{\color{dkorange}}
\def\Bk{\color{black}}

\tikzset{
    state/.style={
                draw,
                rounded corners,
                minimum width=3.4cm,
                minimum height=1cm,
                align=center
            },
    latent/.style={
        minimum width=3.4cm,
        minimum height=1cm,
        align=center
            },
    cstate/.style={
        draw,
        rounded corners,
        minimum width=3.4cm,
        minimum height=1cm,
        align=center,
        fill=Blue!80
        },
    box/.style={
        minimum width=0.5cm,
        minimum height=0.5cm,
        align=center},
    arrow/.style={->, line width=1.2pt, -{Stealth[length=6pt,width=8pt]}},
    arrowrev/.style={<-, line width=1.2pt, {Stealth[length=6pt,width=8pt]}-},
    rotatedtitle/.style={
        rotate=0,
        draw=#1,
        text=#1,
        fill=white,
        fill opacity=0.7,
        line width=1pt,
        rounded corners=2pt,
        inner xsep=8pt,
        inner ysep=4pt,
        font=\bfseries\scriptsize,
        align=center,
        minimum width=15.7cm,
        minimum height=1cm
            }
    }

\title{When vector abundance constrains vector-borne transmission: inside epidemic dynamics within a prescribed population}
\author{Lise Isnel, Julien Papaïx, Olivier Bonnefon, Samuel Soubeyrand, Lionel Roques}
\date{\vspace{-20pt}INRAE, BioSP, 84914 Avignon, France}

\makeatletter
\newtheorem*{rep@theorem}{\rep@title}
\newcommand{\newreptheorem}[2]{%
\newenvironment{rep#1}[1]{%
 \def\rep@title{#2 \ref{##1}}%
 \begin{rep@theorem}}%
 {\end{rep@theorem}}}
\makeatother

\newtheorem{theorem}{Theorem}[section]
\newtheorem{proposition}{Proposition}[theorem]

\newtheorem{remark}{Remark}
\newreptheorem{proposition}{Proposition}

\begin{document}

\maketitle

\paragraph{Abstract}

Since the seminal Ross-Macdonald theory, compartmental models have been extensively used to describe vector-borne pathogen transmission.
While these models often couple vector population dynamics with infection spread, the inherent complexity of vector demographic processes calls for a more modular approach.

Here, we treat the total vector abundance as an externally-prescribed input to a general vector-mediated transmission model. Our objectives are threefold: (i) to provide theoretical guarantees that the prescribed abundance is consistent with the epidemiological model, (ii) to define a local reproduction index informing extinction or emergence near the time of introduction, and to address potential modelling artefact in inter-seasonal propagation, and (iii) to derive insights for diseases transmitted by \textit{Aedes albopictus} in mainland France.
We apply this framework by feeding the epidemiological model with vector-abundance data from entomological surveillance systems or from an external detailed population dynamic model.
Simulations reveal how introduction timing, transmission rate, and susceptible population density shape epidemic extent.

These results highlight the value of entomological surveillance data, and, when unavailable, of their model-based counterparts — particularly relevant for exploring vector-control strategies under projected future climate conditions. This framework has strong potential for applications in animal, human or plant health.


\newpage
\section{Introduction}

\begin{figure}[htbp]
    \centering
    \pgfdeclarelayer{background}
    \pgfdeclarelayer{middle}
    \pgfdeclarelayer{foreground}
    
    \pgfsetlayers{background,middle,main,foreground}
    \vspace*{-20pt}
    \begin{tikzpicture}[
             >=latex,
             node distance=2cm,
             scale=0.9, transform shape
            ]
        
        
        \node[latent, font=\itshape\small] (input_entomo) at (-8.2,7.1)
        {entomological surveillance data\\\includegraphics[width=0.4\textwidth]{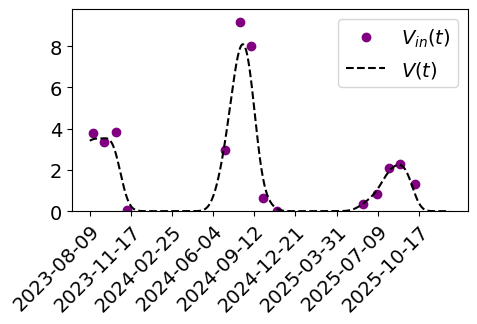}};
        
        \node[latent, font=\itshape\small] (output_entomo_V) at (-8.2,-4.9)
        {\includegraphics[width=0.4\textwidth]{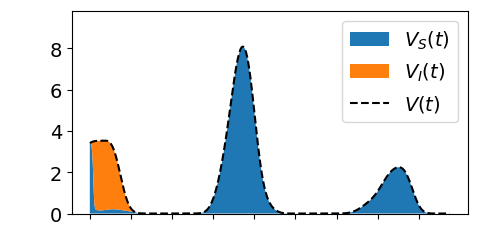}\\epidemiological dynamics\\in the vector compartments};
        
        \node[latent, font=\itshape\small] (output_entomo_H) at (-8.2,-9.4)
        {\includegraphics[width=0.4\textwidth]{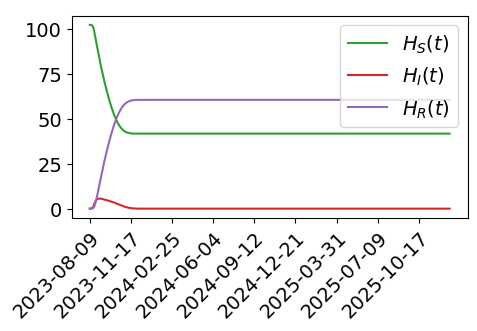}};

        \node[latent, font=\itshape\small] (legend_output_entomo_H) at ($(output_entomo_H.south) + (0,-0.4)$)
        {epidemiological dynamics\\in the host compartments};
            
        \node[latent, text=black!90] (or) at (-4.5,7)
        {OR};
        
        \node[latent, font=\itshape] (input_disti) at (1.4,6.9)
        {\includegraphics[width=0.62\textwidth, clip, trim={0 0 0 0}]{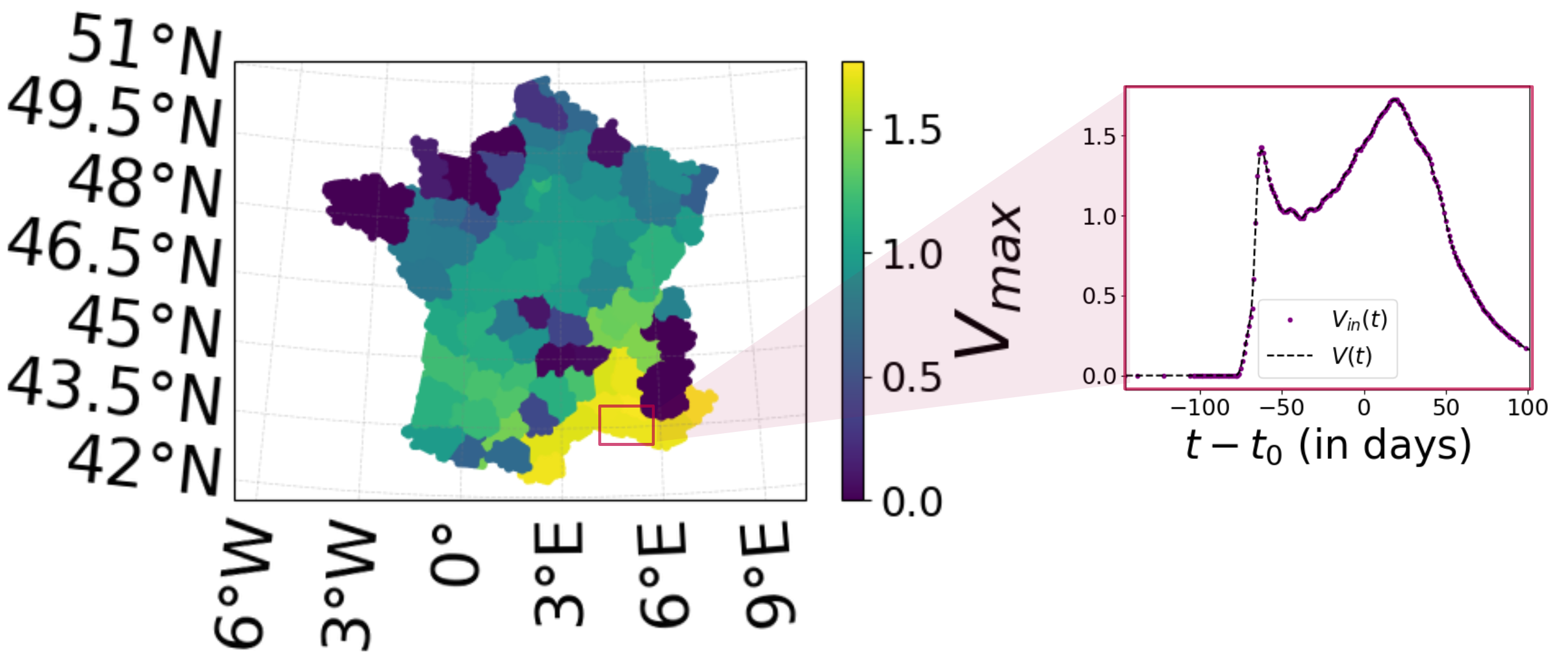}};
        
        \node[latent, font=\itshape\small] (legend_input_disti) at ($(input_disti.north)+(-1.9,0.3)$)
        {output from an external\\vector population dynamic model};
        
        \node[latent, font=\itshape] (output_disti_V) at (1.4,-7.4)
        {\includegraphics[width=0.68\textwidth]{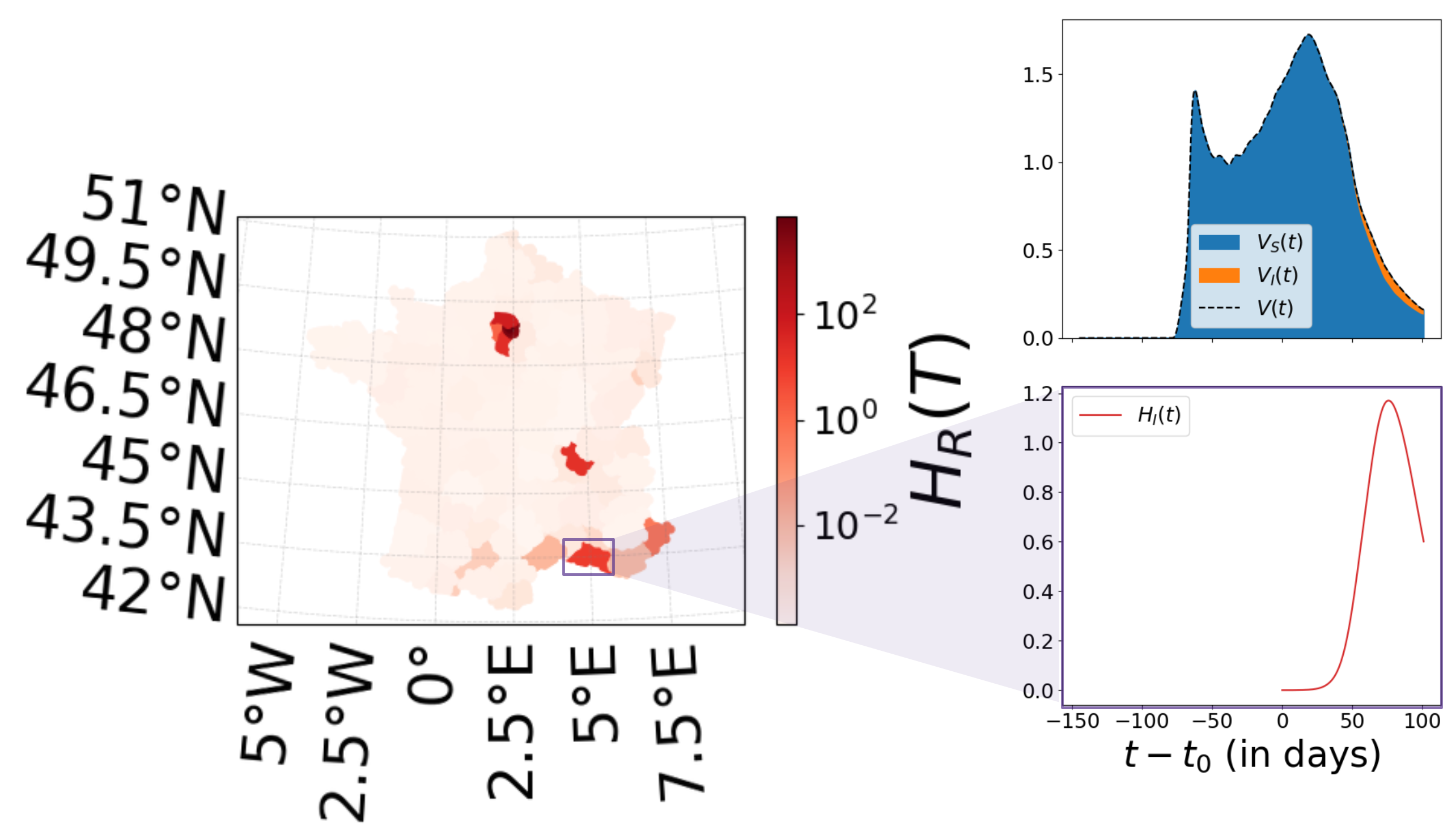}};

        \node[latent, font=\itshape\small] (legend_output_disti_V) at ($(output_disti_V.south)+(-1.9,-0.5)$)
        {total amount of infected cases at $t=T$\\in 94 different territories};


        \node[rotatedtitle=redDark, above=2.55cm of or.north, font=\small\bfseries] (exog_input) {\Large EXOGENEOUS INPUT\\[2pt]Vector population dynamic};

        \node [rotatedtitle=blueDark, below=6.8cm of exog_input.north, font=\small\bfseries] (hostvector_model) {\Large HOST-VECTOR MODEL\\[2pt]compartmental model};
        
        \node [rotatedtitle=purpleDark, below=5.3cm of hostvector_model.north, font=\small\bfseries] {\Large HOST-VECTOR MODEL OUTPUTS\\[2pt]Epidemic / epizootic dynamic};
        
        
        \node[color=redDark, font=\footnotesize\bfseries, align=center] (V) at ($(hostvector_model.south) + (0,-18pt)$) {\Large$V(t)$\\$V_S(t) + V_I(t) = V(t)$};
        
        \def\bodywidth{0.1}
        \def\anglearrowtip{75}
        \pgfmathsetmacro{\lengthtip}{\bodywidth / tan(\anglearrowtip/2)}
        
         \begin{pgfonlayer}{middle}
            \path[
                shade,
                shading=axis,
                top color=redDLight,
                bottom color=blueDLight,
                line width=0.8pt,
                opacity=0.45
                ]
            ($(input_entomo.south)+(-\bodywidth,0.25)$) -- ($(input_entomo.south)+(\bodywidth,0.25)$) -- ($(input_entomo |- V.north west) + (\bodywidth,\bodywidth-0.37)$) -- ($(V.north west) + (-\lengthtip+0.85,\bodywidth-0.37)$) -- ($(V.north west)+(0.85,-0.37)$) -- ($(V.north west) + (-\lengthtip+0.85,-\bodywidth-0.37)$) -- ($(input_entomo |- V.north west) + (-\bodywidth,-\bodywidth-0.37)$) -- cycle;

            \path[
                shade,
                shading=axis,
                top color=redDLight,
                bottom color=blueDLight,
                line width=0.8pt,
                opacity=0.45
                ]
            ($(input_disti.south)+(-\bodywidth-1.5,0.25)$) -- ($(input_disti.south)+(\bodywidth-1.5,0.25)$) -- ($(input_disti |- V.north east) + (\bodywidth-1.5,-\bodywidth-0.37)$) -- ($(V.north east) + (\lengthtip-0.96,-\bodywidth-0.37)$) -- ($(V.north east)+(-0.96,-0.37)$) -- ($(V.north east) + (\lengthtip-0.96,\bodywidth-0.37)$) -- ($(input_disti |- V.north east) + (-\bodywidth-1.5,\bodywidth-0.37)$) -- cycle;

            \path[
                shade,
                shading=axis,
                top color=blueDLight,
                bottom color=purpleDLight,
                line width=0.8pt,
                opacity=0.45
                ]
            ($(V.west) + (-\bodywidth+0.85,-2.91)$) -- ($(V.west) + (+0.85,-2.91 -\lengthtip)$) -- ($(V.west) + (\bodywidth+0.85,-2.91)$) -- ($(V.west |- output_entomo_V.north) + (\bodywidth+0.85,0)$) -- ($(output_entomo_H.west |- output_entomo_V.north) + (-\lengthtip,0)$)  -- ($(output_entomo_V.west -| output_entomo_H.west) + (-\lengthtip,\bodywidth)$) -- ($(output_entomo_V.west) + (-\lengthtip,\bodywidth)$) -- (output_entomo_V.west) -- ($(output_entomo_V.west) + (-\lengthtip,-\bodywidth)$) -- ($(output_entomo_V.west -| output_entomo_H.west) + (-\lengthtip,-\bodywidth)$) -- ($(output_entomo_H.west) + (-\lengthtip,\bodywidth)$) -- ($(output_entomo_H.west) + (0,0)$) -- ($(output_entomo_H.west) + (-\lengthtip,-\bodywidth)$) -- ($(output_entomo_H.west) + (-\lengthtip,-\bodywidth)$) -- ($(output_entomo_H.west) + (-\lengthtip-2*\bodywidth,-\bodywidth)$) -- ($(output_entomo_H.west |- output_entomo_V.north) + (-\lengthtip-2*\bodywidth,2*\bodywidth)$) -- ($(V.west |- output_entomo_V.north) + (-\bodywidth+0.85,2*\bodywidth)$) -- cycle;

            \path[
                shade,
                shading=axis,
                top color=blueDLight,
                bottom color=purpleDLight,
                line width=0.8pt,
                opacity=0.45
                ]
            ($(V.east) + (-\bodywidth-0.89,-2.91)$) --
            ($(V.east) + (-0.89,-2.91-\lengthtip)$) -- ($(V.east) + (\bodywidth-0.89,-2.91)$) --
            ($(V.east |- output_entomo_V.north) + (\bodywidth-0.89,2*\bodywidth)$) --
            ($(output_disti_V.north |- output_entomo_V.north) + (-0.97 + \bodywidth,2*\bodywidth)$) -- ($(output_disti_V.north) + (-0.97 + \bodywidth,2*\bodywidth-0.6)$) --($(output_disti_V.north) + (-0.97,-0.6-\lengthtip)$) -- ($(output_disti_V.north) + (-0.97 - \bodywidth,2*\bodywidth-0.6)$) -- ($(output_disti_V.north |- output_entomo_V.north) + (-0.97 - \bodywidth,0)$) -- ($(V.east |- output_entomo_V.north) + (-\bodywidth-0.89,0)$) --  cycle;

        \end{pgfonlayer}
        
        
        \begin{pgfonlayer}{foreground}
            \node[box] (VS) at ($(V) + (-2,-0.9)$) {$V_S(t)$};
            \node[box] (VI) at ($(V) + (2,-0.9)$) {$V_I(t)$};
        
            \draw[arrow, olive!90] ($(VS.east)+(0.15,0)$) -- ($(VI.west)+(-0.15,0)$) node[midway, below=0.02cm, olive!90] (arrowV) {}; 
        
            \node[box, below=0.7cm of arrowV] (HI) {$\mathbf{H_I(t)}$};
            \node[box, right=2.78cm of HI] (HS) {$H_S(t)$};
            
            \draw[arrow, olive!90] ($(HS.west)+(-0.15,-0.1)$) -- ($(HI.east)+(0.15,-0.1)$) node[midway, above=0.06cm, olive!90] (arrowH) {};
        
            \draw[arrow, olive!40, dashdotted] ($(HI.north)+(0,0.05)$) -- ($(arrowV.center)+(0,-0.05)$);
            \draw[arrow, olive!40, dashdotted] ($(VI.south)-(0,0.05)$) -- ($(arrowH.center)+(0,0.05)$);
            
            \draw ($(VI.north east)+(-0.07,-0.15)$) -- ($(VI.north east) + (-0.07,0.1)$) -- ($(VS.north west)+(0.05,0.1)$) -- ($(VS.north west)+(0.05,-0.15)$);

            \draw[arrowrev, purple!50, densely dashed] ($(VI.east)+(0.7,0.08)$) -- ($(VI.east)+(0,0.08)$) node[midway] (arrowOutVI) {};
            
            \draw[arrow, violet!80, densely dotted] ($(VI.east)+(
            0.7,-0.08)$) -- ($(VI.east)+(0,-0.08)$) node[midway] (arrowInVI) {};

            \coordinate (rVI) at ($(arrowOutVI)!0.5!(arrowInVI) + (2cm,0)$);
            
            \node[violet!1!purple] at (rVI) {}; 
            
            \draw[arrow, purple!50, densely dashed] ($(VS.west)-(0,-0.08)$) -- ($(VS.west)-(0.7,-0.08)$);
        
            \draw[arrow, violet!80, densely dotted] ($(VS.west)+(-0.7,-0.08)$) -- ($(VS.west)+(0,-0.08)$)
                node[midway, violet!1!purple, above=-0.1cm] (r) {};
        
        
        \end{pgfonlayer}

        \begin{scope}[on background layer]
        \node[
            darkgray,
            draw=white,
            fill=white,
            rounded corners=1pt,
            line width=0.01pt,
            align=left,
            inner sep=5pt,
            font=\footnotesize,
            anchor=north west
        ] (legend) at ($(VS.south west)+(-4.9,0.5)$)
        {
            \textbf{Legend}\\[2pt]
        
            \tikz{\draw[
                arrow,
                olive!90,
            ] (0,0) -- (0.5,0);}
            \qquad incidence\\[0.2pt]
        
            \tikz{\draw[
                arrow, olive!40, dashdotted,
            ] (0,0) -- (0.5,0);}
            \qquad cross-contribution to incidence\\[0.2pt]
        
            \tikz{\draw[
                arrow, thick, purple!50, densely dashed, shading=axis,
            ] (0,0) -- (0.5,0);}
            \qquad removal (e.g. mortality)\\[0.2pt]
        
            \tikz{\draw[
                arrow, thick, densely dotted, violet!80,
            ] (0,0) -- (0.5,0);}
            \qquad recruitment (e.g. birth)\\[0.2pt]
            
        };
        \end{scope}
        
        \end{tikzpicture}

    \caption{\textit{Diagram illustrating the prescriptive approach, adopted here, taking a vector population dynamic as an exogenous input to a host-vector compartmental model. Different vector abundance datasets can be used as inputs, for instance entomological surveillance data (left column) or output from an external vector population dynamic model (right column).}}
    \label{fig:graph_intro}
\end{figure}

Classical mechanistic models are typically formulated as closed dynamical systems, where every state variable evolves according to internally prescribed equations.
Yet, in many systems, some components are already observed or reliably modelled in a separate framework. In such cases, it can be natural and advantageous to replace their endogenous dynamics with an observed or independently modelled trajectory, treating it as an external forcing. We refer to this as the \textit{prescriptive approach}, in which the inside dynamic is formulated within a prescribed trajectory. This retains the mechanistic structure of the model while allowing well-characterized components to enter directly through data.

This issue has been recognised in climate modelling, where atmospheric models are commonly driven by prescribed sea-surface temperatures. Wittenberg and Anderson showed that even an error-free prescribed ocean trajectory can lead the forced atmospheric subsystem towards dynamics that are inconsistent with those of the fully coupled system, illustrating that replacing an endogenous component by an external trajectory is not a neutral operation \cite{wittenberg_dynamical_1998}.
The key challenge is then to determine when such a hybrid formulation is mathematically consistent and compatible with the dynamic being modelled.

In other settings, the prescribed quantity is treated as exogenous to the process of interest. In population genetics, for example, genealogical or allele-frequency dynamics can be studied conditional on a specified demographic history or time-varying effective population size \cite{griffiths_sampling_1994,zivkovic_transition_2015, alamil2022}. In these formulations, the prescribed demographic trajectory acts as an external time-dependent input that modulates the genetic process, for instance by determining coalescence rates or the strength of genetic drift, while the genetic state does not dynamically feed back on the population-size trajectory.
Mathematical epidemiology provides another natural setting in which an exogenous process can be prescribed, particularly in compartmental models describing pathogen transmission between hosts and vectors \cite{reiner_systematic_2013}.
Building on the seminal Ross-Macdonald framework, compartmental host-vector models commonly incorporate environmental variables, such as temperature and precipitation, as time-dependent drivers of vector development, survival, activity, or pathogen transmission \cite{greenman_external_2004,huber_seasonal_2018,chowell_assessing_2019,ewing_novel_2021}. Thus, usually, the prescribed quantity is not exogenous to the process of interest. Instead, in this work, the prescribed quantity constitutes a state variable of the system itself.

In their classical form, compartmental host-vector models generally adopt what we refer to here as a \textit{joint approach}: vector demography and pathogen transmission are represented within the same dynamical system. Total vector abundance and its epidemiological compartments are therefore generated jointly by the model. This formulation has generated major theoretical and applied insights and remains central to the analysis of vector-borne diseases \cite{brass_role_2024,ewing_novel_2021,chowell_assessing_2019,demers_dynamic_2018,tuncer_bird_nodate,cousins_modelling_2019,moore_predators_2010,simpson_vector_2012, martin2021}. However, it may be restrictive when vector abundance has already been characterised independently of the epidemiological model. Indeed, entomological surveillance can provide repeated observations of vector population dynamics. For example, the European Food Safety Authority (EFSA) and the European Centre for Disease Prevention and Control (ECDC) share entomological surveillance data for main vectors of pathogen in animal and human health (ticks, mosquitoes, sand flies, and biting midges) \cite{braks_vectornet_2022}. Additional entomological surveillance data are available from national surveillance programmes, research networks, and dedicated vector-monitoring platforms \cite{hammami_spatio-temporal_2025, da_re_vectabundance_2024, leta_updating_2019, zhang_dataset_2019}. Alternatively, vector abundance can be generated by dedicated population-dynamics models that integrate processes such as temperature-dependent development and mortality, photoperiod, diapause, density dependence, dispersal, or landscape structure. Such models also allow prospective approaches that are particularly relevant in the context of global warming and for the evaluation of vector-control strategies. For instance, Pasquali et al. developed a stage-structured mechanistic model of the abundance and seasonal activity of \Aa\ in Europe \cite{pasquali_development_2020}, while Cuéllar et al. predicted the monthly average abundance of \textit{Culicoides} biting midges at a 1 km$^2$ resolution using Random Forests machine learning \cite{cuellar_modelling_2020}. Hence, the prescriptive approach appears more promising because of its modularity, enabling to feed vector-abundance data into the epidemic model and to evaluate the impacts of global warming and of vector-control strategies. Figure \ref{fig:graph_intro} provides an overview of this prescriptive approach applied to host-vector compartmental models.

A closer precedent is provided by Laneri et al., who showed that classical malaria transmission models can be reformulated as non-autonomous systems in which mosquito abundance acts as an external forcing \cite{laneri_forcing_2010}. They subsequently adopted a parsimonious formulation in which vector dynamics are represented implicitly through the force of infection, thereby avoiding explicit representation of mosquito abundance and of susceptible and infected vector compartments. Building on this idea, we ask whether total vector abundance $V(t)$ can be prescribed externally while rigorously retaining explicit susceptible and infectious vector compartments within that population.

The prescribed approach raises a consistency question that does not arise when vector dynamics are represented only implicitly.
At each time, the susceptible and infectious vector compartments must partition the externally prescribed population $V(t)$ while remaining compatible with its trajectory.
Thus, $V(t)$ is not only an input to the epidemiological model but also a time-varying constraint on its admissible state space. The central question is therefore under which conditions an externally prescribed vector abundance trajectory can be embedded into a compartmental host-vector model without violating its mathematical or biological consistency.
Moreover, because $V(t)$ is externally prescribed and generally varies through time, the resulting conditional epidemiological system is non-autonomous. Classical threshold quantities derived for autonomous host-vector models cannot be interpreted as global invasion criteria anymore, and basic reproduction numbers require a specific treatment \cite{bacaer_epidemic_2006}.
Finally, strongly seasonal vector dynamics raise an additional practical difficulty. When vector abundance becomes very low or vanishes during part of the year, deterministic compartmental models may retain arbitrarily small infectious host populations that can subsequently seed renewed transmission when vectors return, producing artificial inter-seasonal persistence in the absence of pathogen reintroduction.

The objective of this work is therefore to provide a rigorous framework for incorporating externally prescribed vector abundance into compartmental host-vector epidemic models. We determine conditions under which a prescribed non-negative trajectory $V(t)$ is compatible with the vector infection dynamics, thereby ensuring mathematical well-posedness and biological admissibility of the resulting conditional system. We then examine the interpretation of epidemic invasion in a time-varying environment, and introduce a low-density correction designed to prevent unrealistic deterministic persistence across periods of vector absence.

We first introduce the main analytical results in biological terms in Section~\ref{sec:motivation}, before developing the full mathematical framework in Section~\ref{sec:well_posed_framework}. We then apply the framework to \textit{Aedes albopictus}-borne chikungunya transmission in mainland France (Section~\ref{sec:Aa-human-model}). Prescribed vector abundance is obtained either from entomological surveillance based on ovitrap counts of \Aa\ eggs or from DISTIGRI, a spatio-temporal model of \Aa\ population dynamics \cite{roques_dynamics_2023}. The application illustrates how externally supplied vector dynamics can be used to reconstruct infection dynamics within the vector population, investigate how introduction timing and vector seasonality shape epidemic outcomes, and explore the differences in epidemiological conditions across locations. It is motivated by the increasing frequency, geographic extent, and magnitude of chikungunya outbreaks in Europe \cite{menegale_risk_2025}.

\section{The challenge of prescribing a vector abundance in epidemic models} \label{sec:motivation}

Four fundamental questions arise when considering embedding a vector abundance trajectory $V(t)$ within a vector-borne disease model:
\begin{enumerate}
    \item what properties must a prescribed vector abundance trajectory
    $V(t)$ satisfy to be used as an epidemiological input?
    \item what constraints must the epidemiological partition of the
    prescribed vector abundance satisfy?
    \item what are the minimal modelling ingredients required to construct a well-posed host-vector model conditioned on a prescribed $V(t)$?
    \item if the amount of adult vectors experiences periodic episodes of absence (e.g. due to winter mortality), how to ensure epidemic fade out if V(t) is sufficiently small on a long enough time period? \end{enumerate}
By addressing these four questions, we aim to provide a rigorous foundation for integrating externally observed or modelled vector dynamics into compartmental models of vector-borne disease transmission. 
In this section, we highlight guiding principles that ensure the resulting models are both biologically relevant and mathematically reliable. The associated mathematical framework is then introduced in Section \ref{sec:well_posed_framework}.

\

In what follows, $V(t)$ denotes the total abundance of epidemiologically active vectors at time $t$. Depending on
the data source and modelling scale, it may represent an absolute number of
vectors, a vector density, or a relative abundance index obtained through a
fixed normalisation. It may originate from empirical time-series observations, such as entomological surveillance data, or from the output of an independent population-dynamics model.   The vector compartments $V_S(t)$ and $V_I(t)$ are
expressed on the same scale as $V(t)$, so that
$V_S(t)+V_I(t)=V(t)$ retains the same interpretation under any of these
representations. The numerical values and units of the transmission
coefficients consequently depend on the scale chosen for $V(t)$. We assume that $V(t)$ is prescribed over a time interval $[t_0,T]$, where $t_0$ denotes the time at which the pathogen is introduced and $T$ the final time horizon considered.

\subsection{What properties must a prescribed vector abundance satisfy?}

The first question concerns the prescribed trajectory itself. To be used as
an input to the host-vector epidemic model, $V(t)$ must satisfy several
regularity and compatibility requirements.

At a minimal level, the prescribed trajectory must satisfy $V(t)\geq0$ and have sufficient temporal regularity for the differential-equation formulation. If $V(t)$ is obtained from a theoretical population model, this regularity may be inherited directly. If $V(t)$ is obtained from empirical observations, however, additional care is needed. Entomological surveillance data are typically sparse, noisy, and collected at discrete times. Direct interpolation may introduce abrupt changes that are incompatible with the vector demography. The prescribed abundance must therefore satisfy not only a regularity
condition but also a compatibility condition with its epidemiological
partition derived below; when necessary, we enforce it through the
regularisation procedure of Section~\ref{apx:data_reg} of the Supplementary material.

\subsection{What constraints must the epidemiological partition satisfy?}
Given a prescribed vector abundance trajectory satisfying the above
requirements, the next step is to identify the constraints that its
epidemiological partition must satisfy.
The central idea is that an externally prescribed vector abundance should not be regarded as an arbitrary external input. Rather, it defines the overall population within which the susceptible and infectious vector compartments must be allotted and evolve. To ensure biological consistency, this constraint is enforced by requiring conservation of the total population together with the non-negativity of all epidemiological compartments.

Decomposing the total vector abundance into two epidemiological states, susceptible vectors $V_S(t)$ and infectious vectors $V_I(t)$, population conservation requires
\begin{equation}\label{eq:Vpop_conservation}
V_S(t)+ V_I(t) =V(t),
\end{equation}
throughout the time interval of interest. 
This constraint expresses the fact that the infection process cannot create vectors outside the prescribed population. Newly infectious vectors can only arise from susceptible vectors already present in the externally imposed abundance $V(t)$. 

In addition, both epidemiological vector compartments must remain non-negative:
\begin{equation}\label{eq:contrainte_positivite}
V_S(t) \ge 0 \quad \text{and}\quad V_I(t) \ge 0.
\end{equation}

Together with equation~\eqref{eq:Vpop_conservation}, this is equivalent to requiring
$$0\leq V_I(t)\leq V(t).$$
If this condition is violated, the model may yield negative numbers of susceptible vectors or more infectious vectors than the total vector population, both of which are biologically inadmissible.

\subsection{What are the minimal ingredients required to construct the conditional model?} \label{subsec:min_vect_subsyst}

Once the prescribed trajectory and its epidemiological partition have been
specified, the question is how to construct the corresponding
host-vector model. The framework developed below relies on a small number of modelling ingredients.

\paragraph{Ensure host subsystem well-posedness.}
First, the host subsystem must be well-posed on the time interval $[t_0,T]$, with existence and non-negativity of host compartments. The specific structure of the host model may vary according to the biological system considered. For example, it may involve a simple susceptible-infectious-recovered structure, or a more detailed description including latent, symptomatic, asymptomatic, immune, or removed host classes.

\paragraph{Define the infection pressure from hosts to vectors.}
Second, the model must specify how infectious hosts generate infectious vectors.
As illustrated by Figure \ref{fig:VB_transmission}, vector-borne diseases are mediated by cyclic pathogen transmission between host and vector populations across two successive infection steps: the infection of a susceptible vector following its encounter with an infected host (belonging to the infectious host substate $H_I(t)$), which, in turn, becomes infectious and likely to transmit the pathogen to susceptible hosts (represented by the susceptible host substate $H_S(t)$).
Transmission is driven by an infection pressure from hosts to vectors, denoted $\lambda_{HV}(H_I)$, which determines the rate at which susceptible vectors become infected through contact with infectious hosts. 
The corresponding incidence of newly infectious vectors is
\begin{equation}\label{eq:incidence_vectors}
\lambda_{HV}  \left(H_I(t)\right)\,V_S(t).
\end{equation}

The infection pressure mainly encodes for the contact rate between infectious hosts and susceptible vectors. It has the dimension of a per-capita rate (time$^{-1}$). 
In classical host-vector models, this contact process may be represented as density-dependent or frequency-dependent \cite{borremans_shape_2017, moore_predators_2010, morel_comparative_2008}. Density-dependent transmission scales with host number or density: $\lambda_{HV}(H_I) = \beta_{hv} H_I$ \cite{cousins_modelling_2019, moore_predators_2010, morel_comparative_2008}. Frequency-dependent transmission scales with frequency of infection in the host population: $\lambda_{HV}(H_I) = \beta_{hv} \frac{H_I}{H}$, where $H$ denotes the total number of hosts \cite{brass_role_2024, ewing_novel_2021, demers_dynamic_2018, moore_predators_2010, simpson_vector_2012, morel_comparative_2008, chamchod_analysis_2011, buonomo_stability_2013, gao_prevention_2016}. More detailed formulations may also incorporate vector activity, host protection measures, or environmental covariates such as land use or vegetation cover.

The framework developed here does not impose a specific functional form for $\lambda_{HV}$; it only requires this infection pressure to satisfy the regularity and positivity assumptions stated in Section \ref{sec:well_posed_framework}.

\paragraph{Specify the remaining processes affecting the infectious-vector compartment.}
Besides host-to-vector transmission, mortality, removal, or vertical
transmission may contribute to changes in the infectious-vector compartment.

We denote by $r(t)$ the net per-capita rate associated with these additional
processes. 
Mortality or other removal processes from the infectious class
contribute negatively to $r(t)$, whereas vertical transmission may contribute
positively. Examples of possible choices for $r(t)$ are given in Appendix~\ref{apx:r(t)}.
The framework developed here does not impose a specific functional
form for $r(t)$. However, as stated in Section~\ref{sec:well_posed_framework}, $r(t)$ is linked
with the prescribed vector abundance through a compatibility condition.
In the particular case where removal from the infectious-vector compartment
occurs only through mortality, this condition means that the prescribed
abundance trajectory cannot decline faster than is compatible with the
assumed mortality process.

In many arthropod applications, mortality is a major component of $r(t)$
and is strongly influenced by temperature. Consequently, temperature may be
particularly relevant when specifying $r(t)$ \cite{delatte_influence_2009, grimaud_modelling_2019, wennergren_population_1993, doeurk_impact_2025}.

\paragraph{Resulting minimal conditional model.}
The constrained dynamics of infectious vectors can then be expressed as$$V_I'(t)
=
\lambda_{HV}  \left(H_I(t)\right)\,V_S(t)
+
r(t)\,V_I(t)$$
where the first term represents host-to-vector transmission and the second
term represents the net contribution of the additional processes introduced
above.

Equivalently, using the population conservation constraint \eqref{eq:Vpop_conservation}, the vector sub-system dynamic can be expressed as
\begin{equation}\label{eq:FV_constrained}
\left\{
\begin{aligned}
V_S'(t) &= V'(t)-V_I'(t) \\
V_I'(t) &= \lambda_{HV}  \left(H_I(t)\right)\,\bigl(V(t)-V_I(t)\bigr)
+
r(t)\,V_I(t)
\end{aligned}
\right.
\end{equation}
where the equation on $V_S(t)$ is not an independent one, but rather ensures that the susceptible and infectious vector compartments always sum to the externally prescribed total abundance.

\begin{figure}[h]
    \centering
    \begin{subfigure}{.44\textwidth}
        \centering
        \includegraphics[width=.99\linewidth]{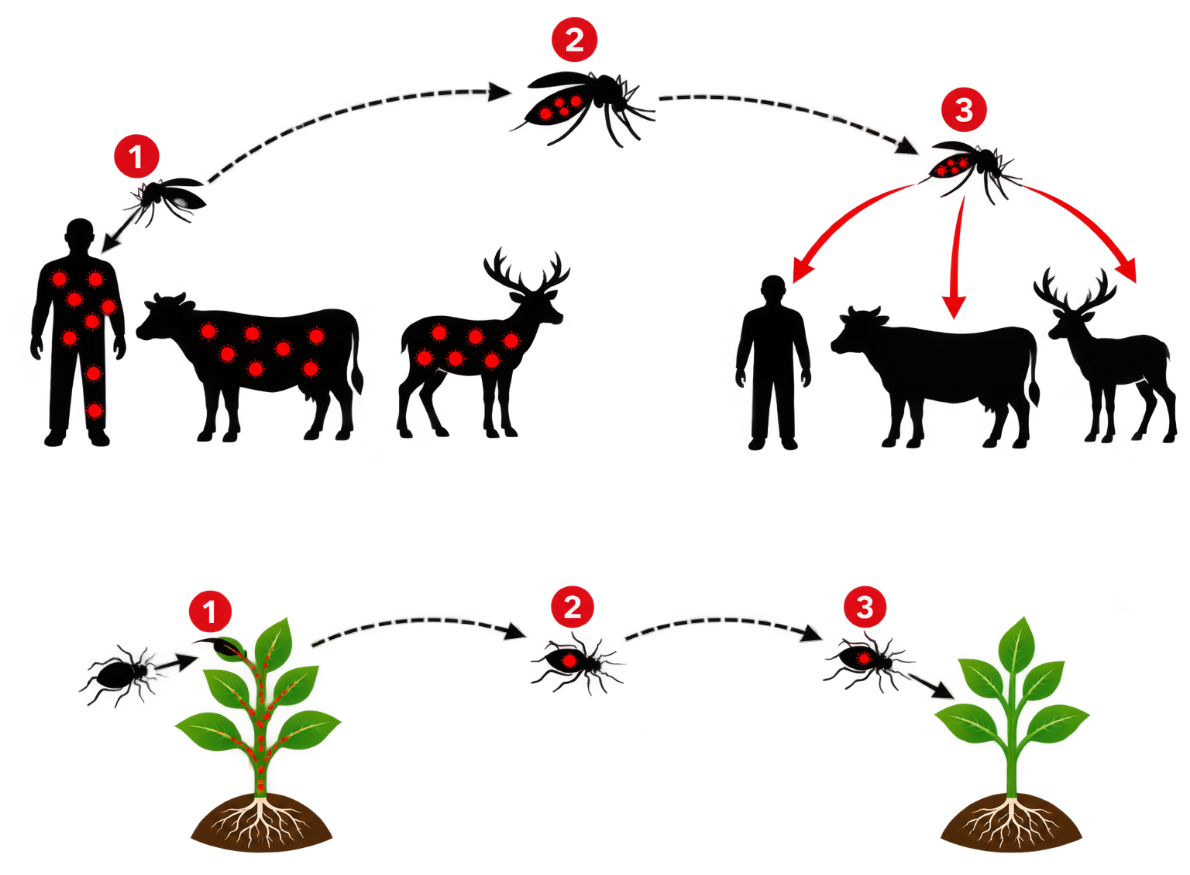}
        \vspace*{0.1cm}
        \caption{}
        \label{fig:HV_schema_icon}
    \end{subfigure}
    \begin{subfigure}{.55\textwidth}
        \centering
        \begin{tikzpicture}[
    scale=0.96, transform shape
    ]
        \node (VI) {$V_I(t)$};
        \node[left=2.75cm of VI] (VS) {$V_S(t)$};
            
        \draw[arrow, olive!90] ($(VS.east)+(0.15,0)$) -- ($(VI.west)+(-0.15,0)$) node[midway, below=0.02cm, olive!90] (arrowV) {\footnotesize $\lambda_{HV}(\mathbf{H_I(t)}) \, V_S(t)$};
            
        \node[below=0.7cm of arrowV] (HI) {$\mathbf{H}_I(t)$};
        \node[right=2.67cm of HI] (HS) {$H_S(t)$};
                
        \draw[arrow, olive!90] ($(HS.west)+(-0.15,-0.1)$) -- ($(HI.east)+(0.15,-0.1)$) node[midway, above=0.06cm, olive!90] (arrowH) {};
            
        \draw[arrow, olive!40, dashdotted] ($(HI.north)+(0,0.05)$) -- ($(arrowV.center)+(0,-0.35)$);
        \draw[arrow, olive!40, dashdotted] ($(VI.south)-(0,0.05)$) -- ($(arrowH.center)+(0,0.05)$);
                
        \draw ($(VI.north east)+(-0.07,-0.15)$) -- ($(VI.north east) + (-0.07,0.1)$) -- ($(VS.north west)+(0.05,0.1)$) -- ($(VS.north west)+(0.05,-0.15)$);

        \node[align=center, font=\large] at ($(VS.north west)!0.5!(VI.north east)+(0,0.6)$) {$V(t)$\\\scriptsize$V_S(t) + V_I(t) = V(t)$};
            
        \draw[arrowrev, purple!50, densely dashed] ($(VI.east)+(0.7,0.08)$) -- ($(VI.east)+(0,0.08)$) node[midway] (arrowOutVI) {};
            
        \draw[arrow, violet!80, densely dotted] ($(VI.east)+(
            0.7,-0.08)$) -- ($(VI.east)+(0,-0.08)$) node[midway] (arrowInVI) {};

        \coordinate (rVI) at ($(arrowOutVI)!0.5!(arrowInVI) + (1.2cm,0)$);
            
        \node[violet!1!purple] at (rVI) {\footnotesize $r(t) \, V_I(t)$};
            
        \draw[arrow, purple!50, densely dashed] ($(VS.west)-(0,-0.08)$) -- ($(VS.west)-(0.7,-0.08)$);
        
        \draw[arrow, violet!80, densely dotted] ($(VS.west)+(-0.7,-0.08)$) -- ($(VS.west)+(0,-0.08)$)
                node[midway, violet!1!purple, above=-0.1cm] (r) {};
        
        \node[left=1.2cm of VS] (Vlabel) {\footnotesize \textbf{Vectors}};
        \node (Hlabel) at ($(Vlabel|-HI)$) {\footnotesize \textbf{Hosts}};
    
        \node[darkgray, align=left] at ($(VS) + (2.2,-4.4)$) {
        \textbf{\footnotesize Legend}\\[2.5pt]
        
        \tikz{\draw[arrow, thick, olive!90] (0,0) -- (0.6,0);} \quad {\footnotesize incidence} \\[0.05pt]
    
        \tikz{\draw[arrow, thick, dashdotted, olive!40] (0,0) -- (0.6,0);} \quad {\footnotesize cross contribution to incidence} \\[0.05pt]
        
        \tikz{\draw[arrow, thick, purple!50, densely dashed, shading=axis,
            left color=blue!80!black,
            right color=blue!20,
            shading angle=0] (0,0) -- (0.6,0);} \quad {\footnotesize removal (e.g. mortality)}\\[0.05pt]
        
        \tikz{\draw[arrow, thick, densely dotted, violet!80] (0,0) -- (0.6,0);} \quad {\footnotesize recruitment (e.g. birth)}
    
        };
    \end{tikzpicture}
    \caption{}
    \label{fig:cycle_vect_transmission}
    \end{subfigure}
    \caption{\textit{\textbf{(a) Schematic representation of vector-borne pathogen transmission}: (1) a susceptible arthropod vector becomes infected following the bite of an infected host, (2) its viral load increases until (3) it reaches levels sufficient for onward transmission to a susceptible host.} \textit{(\textbf{b) Flow chart displaying the basic bricks of a compartmental host-vector model with prescribed vector abundance} $V(t)$}}
    \label{fig:VB_transmission}
\end{figure}

\medskip

\subsection{Under what conditions can we guarantee inter-seasonal epidemic extinction?}\label{subsec:conditions_epidemic_extinction}
When vectors have strongly seasonal dynamics, deterministic compartmental models face an intrinsic limitation. During periods of very low or zero adult-vector abundance, infectious host and vector abundances may become arbitrarily small without reaching zero exactly. These residual deterministic numbers can subsequently seed renewed transmission when vectors return, producing artificial inter-seasonal re-emergence, a manifestation of the \textit{atto-fox problem} \cite{fowler_atto-foxes_2021,mollison_dependence_1991}.

In section \ref{subsec:shrinking-adj}, we introduce an adjustment that guaranties epidemic extinction. Indeed, we show that, in the absence of pathogen reintroduction, a sufficiently severe inter-seasonal decline in vector abundance eliminates the risk that infinitesimally small infection levels could trigger artificial outbreak in the subsequent season.
Moreover, we show that, under favourable conditions, this adjustment still allows epidemic growth from a biologically meaningful introduction.

\medskip
\medskip

Overall, these requirements are not specific to a particular vector species or pathogen. They come into play when the vector abundance is generated externally and subsequently used as input to the transmission model.

The following section formalises this construction and gives conditions ensuring that the infectious vector compartment remains non-negative and bounded above by the prescribed vector abundance. We then show how this framework can be applied to \Aa-borne disease transmission using either surveillance-derived vector abundance or abundance predicted by an independent population dynamic model.

\section{A well-posed framework for infection dynamics within a prescribed vector abundance} \label{sec:well_posed_framework}

\subsection{General constrained host-vector system}
A general deterministic host-vector model without delay with an SI structure for the vector subsystem satisfying the minimal hypothesis defined in Section \ref{subsec:min_vect_subsyst} can be written as

\begin{equation}\label{eq:model_gal}
\left\{
\begin{aligned}
\mathbf{H}'(t) &= \mathbf{F_H} \left(t,\mathbf{H}(t), V_S(t), V_I(t)\right),\\
V_S'(t) &= V'(t)-V_I'(t), \\
V_I'(t) &= \lambda_{HV}  \left(\mathbf{H}_I(t)\right)\,\bigl(V(t)-V_I(t)\bigr)
+
r(t)\,V_I(t)
\end{aligned}
\right.
\end{equation}
where $\mathbf{H}(t) = (H_S(t), \mathbf{H}_I(t), \mathbf{H}_R(t))$ denotes the epidemiological state of the host population composed of susceptible, uninfected hosts $H_S(t) \in \mathbb{R}_+$, an infectious host sub-state $\mathbf{H}_I(t) \in \mathbb{R}_+^n \,, n \in \mathbb{N}^*$ gathering all infectious host compartments, and a remaining host sub-state $\mathbf{H}_R(t) \in \mathbb{R}_+^m \,, m \in \mathbb{N}$ collecting compartments of individuals that are neither susceptible nor infectious.

In this formulation, the host dynamics are kept as general as possible and are described by a host field $\mathbf{F_H}$.
Therefore, while keeping the host dynamics generic and adjustable, we constrain the vector subsystem so that the evolution of $V_S$ and $V_I$ is consistent with the population conservation constraint \eqref{eq:Vpop_conservation}, and, the transfer from $V_S$ to $V_I$ reflects an epidemiological infection process.

We establish below conditions ensuring that system \eqref{eq:model_gal} is well-posed and that the vector infection dynamics remains compatible with the prescribed abundance $V(t)$, in particular that \eqref{eq:Vpop_conservation} and \eqref{eq:contrainte_positivite} hold for all $t\in[t_0,T]$.

\subsection{Existence and uniqueness} \label{par:well-pos}
We first state assumptions ensuring existence and uniqueness of the solution:
\begin{enumerate}[label=(A\arabic*)]
    \item $V\in C^1([t_0,T])$, $V(t)\ge 0$ for all $t\in[t_0,T]$;
    \label{item:V_C1_positive}
    \setcounter{enumi}{1}
    \item $\mathbf{F_H}$ is continuous in its first variable and locally Lipschitz continuous in its second, third and fourth variables;
    \label{item:F_H_Lz}
    
    \item $\lambda_{HV}$ is locally Lipschitz continuous;
    \label{item:lambda_HV_Lz}
    
    \item $r$ is continuous;
    \label{item:r_continuous}
\end{enumerate}
Under assumptions \cref{item:V_C1_positive,item:F_H_Lz,item:lambda_HV_Lz,item:r_continuous}, the Cauchy-Lipschitz theorem gives a unique local solution to system \eqref{eq:model_gal} for every initial condition
\begin{equation}
\label{eq:CI_model}
\left\{
\begin{aligned}
\mathbf H(t_0)
&=
\left(
H_{S,0},
\mathbf H_{I,0},
\mathbf H_{R,0}
\right)
\in
\mathbb R_+
\times
\mathbb R_+^n
\times
\mathbb R_+^m,
\\
0
&\leq
V_I(t_0)
\leq
V(t_0),
\\
V_S(t_0)
&=
V(t_0)-V_I(t_0).
\end{aligned}
\right.
\end{equation}
The solution is defined on a maximal interval $[t_0,\overline t)$, with
$\overline t\leq T$ if finite-time blow-up occurs. If, in addition, the
right-hand side satisfies an appropriate global Lipschitz or global growth condition on the
relevant domain, the solution extends throughout $[t_0,T]$
\cite{hartman_2002}.

In the applications below $\overline t = T$, and pathogen introduction is usually imposed through
the host population, with $\mathbf H_{I,0}\neq\mathbf0$ and $V_I(t_0)=0$.
The general formulation above also allows introductions involving infectious
vectors.

\subsection{Biological admissibility of the vector compartments}
\label{subsec:assumptions_bio_admissibility}

Existence and uniqueness alone do not guarantee that the solution is biologically meaningful. Since $V(t)$ is prescribed externally, we must also ensure that the infectious-vector compartment remains within the prescribed vector abundance envelope:

\begin{equation}\label{eq:VI_constraint}
    0\leq V_I(t)\leq V(t)
\quad\text{for all }t\in[t_0,T].
\end{equation}

\medskip

Thus, in addition to assumptions \cref{item:V_C1_positive,item:F_H_Lz,item:lambda_HV_Lz,item:r_continuous}, requiring
\begin{enumerate}[label=(A\arabic*)]
\setcounter{enumi}{4}
    \item $\lambda_{HV}(\mathbf H_I)\geq0$ for every biologically admissible infectious-host state $\mathbf H_I$; and
    \label{item:lambda_HV_positive}
    
    \item $V'(t)-r(t)V(t)\geq 0$ for all $t\in[t_0,T]$,
    \label{item:V_r_compatible}
\end{enumerate}
is sufficient for \eqref{eq:VI_constraint} to hold as stated in Section~\ref{apx:critera_A4_A5} of the Supplementary material.

Assumption~\ref{item:lambda_HV_positive} states that host-to-vector transmission contributes non-negatively to the production of infectious vectors. 
Assumption~\ref{item:V_r_compatible} is a compatibility condition between the prescribed vector abundance $V(t)$  and its epidemiological partition into susceptible
and infectious vectors. Indeed, when $V_I(t)=V(t)$,
equivalently $V_S(t)=0$, one has
$
V_S'(t)=V'(t)-r(t)V(t),
$ from~\eqref{eq:FV_constrained}. Hence,
Assumption~\ref{item:V_r_compatible} is necessary at this boundary to ensure
that $V_S$ remains non-negative, and is sufficient to ensure that
$V_I(t)\leq V(t)$ for every solution starting from an admissible initial
condition, as detailed in Section ~\ref{apx:critera_A4_A5} of the Supplementary material.

In the application considered below, we assume that there is no vertical transmission, that
infectious vectors remain infectious throughout their lifespan, and that
susceptible and infectious adult vectors experience the same per-capita
mortality rate $\mu(t)>0$. Consequently, 
$r(t)=-\mu(t).$ 
Assumption~\ref{item:V_r_compatible} then becomes
$$V'(t)+\mu(t)V(t)\geq 0.$$
Under these assumptions, this condition means that the prescribed total vector abundance cannot decline faster than what is compatible with the assumed mortality process. If this compatibility condition is not satisfied by a prescribed trajectory,
we enforce it by regularising this trajectory using the procedure described in
Section~\ref{apx:data_reg} of the Supplementary material.

\medskip


\subsection{A local reproduction index at the time of introduction}

The model introduced above is non-autonomous because the prescribed vector
abundance $V(t)$ and, more generally, the demographic and environmental
quantities entering the transmission model may vary with time. Consequently,
the classical basic reproduction number \cite{diekmann_definition_1990} of an autonomous system is not
directly defined for the full model over the entire time interval. 

Nevertheless, it is useful to characterise the epidemiological conditions
prevailing at the time of pathogen introduction $t_0$. To this end, we
consider an auxiliary autonomous system obtained by freezing the
time-dependent quantities at their values at $t_0$: $V(t)=V(t_0)$, $r(t) = r(t_0)$, and $F_H(t,\cdot) = F_H(t_0,\cdot)$. We refer to the
reproduction number of this frozen system as the \emph{local reproduction
index} and denote it by
\begin{equation}
\mathcal R_l(t_0).
\end{equation}
This quantity characterises the local potential for epidemic growth in the
epidemiological environment encountered at the time of introduction. It
does not account for subsequent changes in vector abundance or other
time-dependent quantities.

Using the next-generation matrix method, the pathogen's basic reproduction number is obtained as the spectral radius of the next-generation matrix \cite{van_den_driessche_reproduction_2002, brauer_mathematical_2008}. 
In the particular case where transmission is strictly vector-mediated and alternates between hosts and vectors, with no within-host, no within-vector nor vertical transmission, this matrix has an off-diagonal host-vector structure. In that case, the basic reproduction number classically writes as the geometric mean of the numbers of secondary infections produced across the two host-vector directions \cite{wonham_epidemiological_2004, moore_predators_2010}:
\begin{equation} \label{eq:Ro_gnrl}
    \mathcal{R}_l(t_0) = \sqrt{\mathcal{R}_{l, VH}(t_0) \, \mathcal{R}_{l, HV}(t_0)}
\end{equation}

\noindent Here, $\mathcal{R}_{l,VH}(t_0)$ denotes the number of infected hosts produced by one infectious vector in the frozen epidemiological environment at time $t_0$, whereas $\mathcal{R}_{l,HV}(t_0)$ denotes the number of infectious vectors produced by one infectious host. The square root arises because, in a strictly alternating vector-mediated transmission cycle, two transmission generations are required for an infectious host to produce new infectious hosts; the same reasoning applies to infectious vectors.
If additional transmission routes are present, such as vertical transmission among vectors, the appropriate quantity is the spectral radius of the full next-generation matrix. Such more general situations are discussed in Appendix~\ref{apx:r(t)}.
Moreover, since explicit expressions are model-specific, the formula is retained in a generic form. An application is presented in Section \ref{sec:Aa-human-model}.


\subsection{Extinction and emergence of the pathogen near the time of introduction}

For the frozen autonomous system used to define
$\mathcal R_l(t_0)$, the usual next-generation-matrix threshold has a local
stability interpretation 
\cite{van_den_driessche_reproduction_2002}. In particular,
$\mathcal R_l(t_0)<1$ implies local asymptotic stability of the disease-free
state in the infected directions, whereas $\mathcal R_l(t_0)>1$ implies loss
of this local stability. The case $\mathcal R_l(t_0)=1$ is not determined by
this linear argument.

This interpretation is local: the next-generation matrix is built from the
linearisation at the disease-free state and therefore describes the behaviour
of sufficiently small introductions at the time of introduction $t_0$. For non-linear transmission terms, the
first-order contribution is captured by their derivative at the disease-free
state. If this derivative vanishes, higher-order nonlinear effects may still
allow epidemic growth following the introduction of a sufficiently large
number of infectious individuals, even though the disease-free state is
locally stable.

Moreover, for the original non-autonomous system,
$\mathcal R_l(t_0)$ characterises only the epidemiological environment frozen
at the introduction time. It does not account for subsequent variations in
$V(t)$, $r(t)$, or other time-dependent coefficients. It should therefore be
interpreted as a local reference index, rather than as a necessary and
sufficient criterion for pathogen persistence or epidemic size over
$[t_0,T]$.

\subsection{Seasonal fade-out and realistic inter-seasonal propagation}\label{subsec:shrinking-adj}

We now propose a framework that ensures realistic inter-seasonal propagation of a pathogen.
To eliminate the \textit{atto-fox} artefact introduced in Section~\ref{subsec:conditions_epidemic_extinction}, one can adopt a discrete or stochastic formulation \cite{breban_role_2009,thompson_will_nodate,fowler_atto-foxes_2021}.
Here we retain a deterministic framework and replace the standard linear host-to-vector infection pressure by a low-density attenuation, referred to as the \emph{shrinking adjustment}.

As an example, considering a classical SIR structure for the host population ($n=1$, $m=1$), the shrinking adjusted infection pressure associated with the density-dependent transmission, $\beta_{hv} H_I$, can be
\begin{equation}
\label{eq:lambdahv_atto}
\lambda_{HV}(H_I)
=
\beta_{hv}g(H_I),
\qquad
g(H_I)
=
\begin{cases}
0,
& \text{if } H_I\leq0,
\\[4pt]
\dfrac{H_I^2}{H_\star},
& \text{if } 0<H_I<H_\star,
\\[6pt]
H_I,
& \text{if } H_I\geq H_\star.
\end{cases}
\end{equation}
where $H_\star$ corresponds to the infected host threshold under which shrinking adjustment is applied.

The factor $H_\star^{-1}$ preserves the physical dimension of $H_I$, while
\begin{equation}
\frac{\lambda_{HV}(H_I)}{H_I}
\longrightarrow0
\qquad
\text{as } H_I\to0^+,
\end{equation}
so host-to-vector transmission is attenuated more strongly than under the
standard linear formulation at very low infectious-host density.

Because $g'(0)=0$, the disease-free set of the frozen shrinking-adjusted
system is locally attracting in the infected directions, independently of
the value of the reproduction index of the corresponding linear model.
Importantly, this local stability does not preclude epidemic growth when the
initial infection level is sufficiently far from zero. Section~\ref{apx:emgc_ext_under_shrinking} of the Supplementary material shows that, when one infectious host is introduced, there exist explicit epidemiological conditions under which the infectious
host compartment necessarily undergoes positive growth.
Thus, the shrinking adjustment suppresses transmission generated by
arbitrarily small residual densities without preventing epidemic growth from
a biologically meaningful finite introduction when conditions are
sufficiently favourable.

More precisely, Section~\ref{apx:emgc_ext_under_shrinking} of the Supplementary material considers the epidemiological environment frozen at $t_0$. It shows that
$\mathcal R_l(t_0)<1$, where $\mathcal R_l$ is given by
equation~\eqref{eq:Ro_gnrl} and denotes the reproduction index of the
corresponding \textit{linear} frozen model, remains a sufficient condition for
extinction under the shrinking adjustment. Conversely, the sufficient condition derived for positive growth of the
infectious host compartment following the introduction of one infectious host
necessarily implies $\mathcal R_l(t_0)>1$, but $\mathcal R_l(t_0)>1$ alone
is not sufficient to guarantee such growth. Hence, $\mathcal R_l$ is used
only as a reference index and is not a local threshold for the
shrinking-adjusted system.

Section~\ref{apx:proof_winter_extinction} of the Supplementary material then addresses the full
non-autonomous system and formalises the inter-seasonal mechanism. For every
sufficiently long low-vector period, it provides a sufficiently small upper
bound on $V(t)$ such that, if vector abundance remains below this bound
throughout the period, the infectious host and vector compartments enter a
shrinking-dominated invariant region. Once this region has been reached,
 both $H_I(t)$ and $V_I(t)$ subsequently converge
exponentially to zero. Crucially, this decay continues even if vector
abundance later increases again. Therefore, in the absence of pathogen reintroduction, a
sufficiently severe inter-seasonal reduction in vector abundance prevents
arbitrarily small deterministic infection levels left from one vector season
from initiating an artificial outbreak in the next.

\section{Applying the framework to \Aa-human arbovirus model with prescribed vector abundance}\label{sec:Aa-human-model}
We now introduce a fully specified host-vector model satisfying the well-posedness conditions established in Section \ref{sec:well_posed_framework}. The model is designed to describe arbovirus transmission between \Aa\ and humans within a spatial domain $\Omega$ (typically, a city or a region).

\subsection{\Aa-human arbovirus model}
On the host side, we consider a classical SIR structure without demographic turnover on the epidemic time scale and with a finite total population size. The host state is written
$\mathbf{H(t)} = (H_S(t), H_I(t), H_R(t)) \in \mathbb{R}_+^3,$
where $H_S$, $H_I$ and $H_R$ denote the density (i.e. the number of individuals per unit area of the spatial region under consideration) of susceptible, infectious and recovered hosts, respectively. Transmission from vectors to hosts is assumed density-dependent, with force of infection
$\lambda_{VH}(V_I(t)) = \beta_{vh} V_I(t)$,
where $\beta_{vh}>0$ is the baseline transmission coefficient from infectious vectors to susceptible hosts.

\medskip 
On the vector side, as a first step, we consider the infection pressure to be a standard linear form:
\begin{equation} \label{eq:lambdahv}
\lambda_{HV}  \left(H_I\right)
=
\beta_{hv} H_I,
\end{equation}
allowing the computation of the frozen-environment local reproduction index  and the derivation of conditions for invasion. In \eqref{eq:lambdahv}, $\beta_{hv}>0$ denotes the baseline transmission coefficient from infectious hosts to susceptible vectors.
Moreover, we assume that there is no vertical transmission, that infectious vectors remain
infectious throughout their lifespan, and that susceptible and infectious
adult vectors experience the same per-capita mortality rate $\mu(t)>0$. Consequently,
\[
r(t) = -\mu(t).
\]
Hence, assumption~\ref{item:V_r_compatible} becomes
\begin{equation}
V'(t)+\mu(t)V(t)\geq0.
\end{equation}

\medskip

\medskip

The dynamic then writes
$$  \left\{
    \begin{aligned}
    H_S'(t) &= - \beta_{vh} V_I(t) H_S(t) \\
    H_I'(t) &= \beta_{vh} V_I(t) H_S(t) - \gamma H_I(t) \\
    H_R'(t) &= \gamma H_I(t) \\
    V_S'(t) &= V'(t)-V_I'(t) \\
V_I'(t) &= \beta_{hv} H_I(t)\,\bigl(V(t)-V_I(t)\bigr)
-\mu(t)\,V_I(t)
    \end{aligned}\right.,
$$

\noindent where $\gamma>0$ is the recovery rate ($1/\gamma$ is the mean duration of the host infectious stage).

Because $H_S$, $H_I$, and $H_R$ are densities, host outcomes reported below
as numbers of individuals are obtained by multiplying by the spatial area.
We therefore write
\begin{equation}
N_X(t):=|\Omega|H_X(t),
\qquad
X\in\{S,I,R\}.
\label{eq:host_density_to_number}
\end{equation}
When $H_R(t_0)=0$, $N_R(T)=|\Omega|H_R(T)$ is the cumulative number of hosts
that have passed through the infectious compartment by time $T$.

\paragraph{Model well-posedness.}
This host-vector formulation satisfies assumptions \cref{item:F_H_Lz,item:lambda_HV_Lz,item:lambda_HV_positive,item:r_continuous} as verified in Section~\ref{apx:well-pos-example} of the Supplementary material.
Regarding assumptions \ref{item:V_C1_positive} and \ref{item:V_r_compatible}, when an empirical or simulated trajectory does not satisfy them directly, we apply the Gaussian regularisation described in
Section~\ref{apx:data_reg} of the Supplementary material, selecting a smoothing scale that enforces the compatibility inequality.

\paragraph{Local reproduction index.}
Appendix~\ref{apx:r(t)} gives the next-generation matrix of
the epidemiological environment frozen at $t_0$. With
\begin{equation}
V_0:=V(t_0),
\qquad
\mu_0:=\mu(t_0),
\qquad
H_{S,0}:=H_S(t_0),
\end{equation}
the resulting local reproduction index is
\begin{equation}
\mathcal R_l(t_0)
=
\sqrt{
H_{S,0}V_0
\frac{
\beta_{hv}\beta_{vh}
}{
\gamma\mu_0
}
}.
\label{eq:Rl_application}
\end{equation}
Equivalently,
\begin{equation}
\mathcal R_l(t_0)
=
\sqrt{
\mathcal R_{l,VH}(t_0)
\mathcal R_{l,HV}(t_0)
} \hbox{ with }
\mathcal R_{l,VH}(t_0)
=
\frac{
\beta_{vh}H_{S,0}
}{
\mu_0
},
\qquad
\mathcal R_{l,HV}(t_0)
=
\frac{
\beta_{hv}V_0
}{
\gamma
}.
\end{equation}
Thus $\mathcal R_l(t_0)$ increases with the susceptible-host density,
the prescribed vector abundance at introduction, and the two transmission
coefficients, and decreases with host recovery and vector mortality.

\subsection{Data and parameters}
We hereafter apply this conditional approach with two different prescribed total vector populations $V(t)$ as inputs: one obtained from entomological surveillance of \Aa, the other from an external vector model.

\paragraph{Entomological data from surveillance of \Aa.}\Aa\ mosquitoes are part of the Culicidae family. After mating, female mosquitoes search for blood meals needed for egg growth and egg-laying. The eggs are then laid on wet places and develop for 2 to 7 days in the aquatic phase \cite{abd_life_2020, kweka_ecology_2018}. Depending on environmental conditions, these eggs might survive, hatch, and undergo the three remaining life stages: larva, pupa, and adult. The adults then disperse and reproduce in turn. If European winter temperature conditions do not allow adults to survive \cite{delatte_influence_2009}, \Aa\ produces diapausing winter eggs that hatch during the next spring to maintain its presence \cite{roques_dynamics_2023}. The difficulty of accurately modelling the mosquito life cycle justifies this approach, as it permits demographic processes and epidemic dynamics to be modelled separately.

The observed \Aa\ egg counts were obtained from entomological surveillance programs deployed in four different cities. In Nice, France, \Aa\ eggs were monitored throughout 2025 using a network of 47 ovitraps. In Bergerac, Decize and Dijon, France, eggs were collected using ovitraps deployed in urban areas between 2017 and 2025, with five ovitraps in Bergerac and two in each of the other cities.
In both cases, the collected eggs were subsequently identified by trained entomologists.

Each ovitrap consisted of a three litre black plastic bucket filled with a plant maceration infusion and covered with a wire mesh. A polystyrene oviposition support was placed inside the trap to provide a suitable substrate for egg laying. Ovitraps were installed at ground level under vegetation cover (e.g., bushes or shrubs).
The traps were deployed, maintained, and collected by a public surveillance operator. Following collection, eggs were identified under a binocular microscope by an entomologist.

The pre-processing steps applied on these observed data consists in a piecewise interpolation (adding zeros outside of the seasonal activity period of \Aa\ where no trap was set), followed by the regularisation proposed in Section~\ref{apx:data_reg} of the Supplementary material, that ensures well-posedness criteria are satisfied.
As a first approximation, we assume that the number of eggs is proportional to the number of females involved in vector-mediated transmission (i.e. the number of \textit{potentially infectious females})
$$
\underbrace{n_f(t)}_{\shortstack{\footnotesize total number\\\footnotesize of potentially\\\footnotesize infectious females}}
=
\underbrace{n_e(t)}_{\shortstack{\footnotesize total\\\footnotesize number \\\footnotesize of eggs}}
\,\cdot
\underbrace{(1-\mu_l(t))}_{\shortstack{\footnotesize larval\\\footnotesize survival}}
\,\cdot
\underbrace{\frac{1}{2}}_{\shortstack{\footnotesize sex\\\footnotesize ratio}}.
$$
$\mu_l$ denotes the larval mortality and is presented in Section~\ref{apx:larval_survival} of the Supplementary material.
The resulting adult-female abundance proxy, rescaled by the city area, is used as the prescribed vector-abundance input (i.e. $V(t) = n_f(t)/|\Omega|$).

\paragraph{Entomological data from an external vector model.}
Simulated data arise from DISTIGRI, a spatio-temporal model of the current and future expansions at intra- and inter-annual scales of \Aa\ \cite{roques_dynamics_2023}.
It models adult \Aa\ life cycle and dispersal over the French territory, with faster dispersal along the main roads, and taking into account both temperatures and photoperiod data at a high spatial and temporal resolution.
In this study, DISTIGRI is used to estimate the mean density of adult \Aa\ over one year (2024) in each mainland French department (corresponding to 94 administrative territories of median surface area of 5,964 km$^2$).
The initial conditions are defined based on the observed presence of \Aa\ in March 2022 summarized in ECDC distribution map \cite{ECDC_Aedes_map}.
The temperature data at the hourly  time step on a $8 \times 8 \mbox{ km}^2$ grid covering France were provided by the SAFRAN database over the period ranging from August 1$^{st}$, 2021 to July 31$^{st}$, 2025 \cite{durand_reanalysis_2009}.\newline
For the epidemiological application, the DISTIGRI adult-density output is
treated as the density of females potentially participating in transmission. $V(t)$
should be interpreted as a relative vector-abundance index; comparisons
between territories are then meaningful on this common scale, whereas the
numerical values of the fitted transmission coefficients remain conditional
on DISTIGRI normalisation.

\paragraph{Vector mortality rate.} In our setting, $\mu(t)$ is temperature dependent. As presented in Section~\ref{apx:T} of the Supplementary material, $\mu$ is calibrated from published experimental data on adult female \Aa\ mortality across a range of temperatures.
When using entomological surveillance data as input, temperature data at the daily time step at the French department scale were provided by an open government dataset \cite{T_daily_opendata}.
When using DISTIGRI data as input, temperature data at the hourly time step on a 8 $\times$ 8 km$^2$ grid covering France were provided by the SAFRAN database \cite{durand_reanalysis_2009}.
In either case, the temperature series was smoothed with a Gaussian kernel
before evaluating $\mu(t)$, yielding a continuous mortality-rate trajectory
as required by assumption~\ref{item:r_continuous}.

\paragraph{Epidemiological data and parameters.}
The human infectious period for chikungunya virus is assumed to be approximately 5 days ($\gamma = 1/5$ days$^{-1}$) \cite{ecdc_factsheet_2026, SpF_fiche_chik}.
The date of introduction and the cumulative number of infections in the human population over the epidemic period were obtained from Santé publique France epidemiosurveillance program \cite{SpF_bulletin_2025}.
The transmission parameters ($\beta_{hv}$, $\beta_{vh}$) were estimated from numerical simulations (see Section \ref{subsec:city}).

\paragraph{Population density and superficie area}of French cities and administrative territories arise from Insee data \cite{insee_recensement_2023}.

\paragraph{Shrinking adjustment.}
When shrinking adjustment is introduced, we move from the standard linear form of the infection pressure (see equation \ref{eq:lambdahv}) to the shrinking adjusted infection pressure (see equation \ref{eq:lambdahv_atto}). Moreover, we set
\begin{equation}
H_\star=\frac{1}{|\Omega|},
\end{equation}
so that $H_I=H_\star$ corresponds to one infectious host in the spatial
domain $\Omega$, and the adjustment acts below this scale.

\paragraph{Numerical simulations.} We used an implicit Crank-Nicolson time-stepping scheme. The nonlinear terms were treated with a Newton-Raphson algorithm. Time integration was performed from $t_0$ to $T$ with a prescribed time step $\Delta t$ (typically of order $10^{-2}$). Time-dependent model inputs ($V(t)$, $V'(t)$ and $r(t)$) were evaluated on the simulation time grid before integration and subsequently used at the corresponding time steps.


\begin{figure*}[hb!]
        \subfloat[]{%
            \includegraphics[width=.45\linewidth]{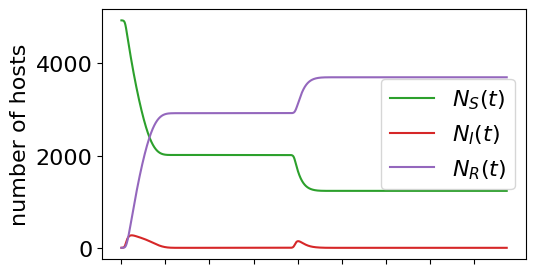}%
            \label{fig:output_HVModel_all}%
        }\hfill
        \subfloat[]{%
            \includegraphics[width=.45\linewidth]{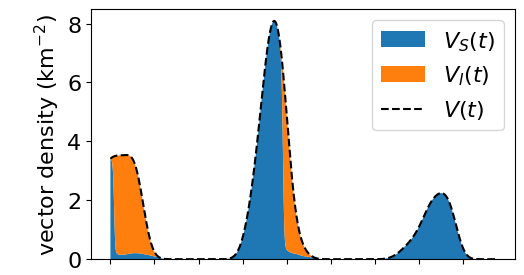}%
            \label{fig:output_HVModel_V}%
        }\\
        \subfloat[]{%
            \includegraphics[width=.45\linewidth, clip, trim={0 5 0 0}]{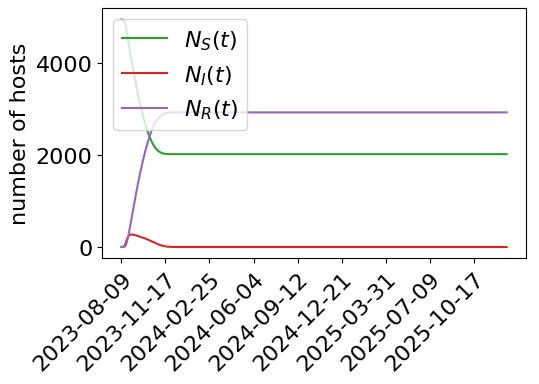}%
            \label{fig:output_sr}%
        }\hfill
        \subfloat[]{%
            \includegraphics[width=.45\linewidth, clip, trim={0 5 0 0}]{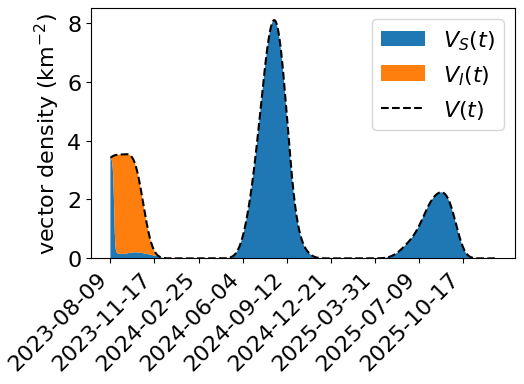}%
            \label{fig:output_sr_V}%
        }
        \caption{\small\textit{\textbf{Outputs of the \Aa-human arbovirus model} (a)-(b) with the standard linear form infection pressure, or (c)-(d) with shrinking adjustment.\\
        (a),(c): evolution of the number of susceptible, infectious and recovered hosts ($N_S, N_I, N_R$) after the introduction of the pathogen at $t_0 =$ August $9^{th}$, 2023.\\
        (b),(d): corresponding distribution of the susceptible vectors $V_S(t)$ and the infectious vectors $V_I(t)$ within the prescribed vector abundance $V(t)$.\\
        Specified vector abundance: entomological surveillance data of eggs of \Aa\, collected in Decize, France, between 2023 and 2025. $\beta_{vh}=0.004$, $\beta_{hv}=0.4$, $\gamma = 1/5$ days$^{-1}$.\\}}
        \label{fig:output_HVModel}
    \end{figure*}

\newpage
\subsection{Results from prescribing entomological surveillance data}

\paragraph{Reconstructing inter-seasonal infection dynamics within the prescribed vector abundance.}
The host-vector model describes the internal infection dynamics within both host and vector populations. Figure \ref{fig:output_HVModel} displays this dynamic after the introduction of one infected host, in a favourable emergence context ($\mathcal{R}_l(t_0) >1$), and driven by entomological surveillance data of \Aa\, in Decize, France, between 2023 and 2025.
Figures \ref{fig:output_HVModel_all} and \ref{fig:output_HVModel_V} illustrate the model output obtained with the standard linear form of the infection pressure (see equation~\eqref{eq:lambdahv}).
For the vectors, Figure \ref{fig:output_HVModel_V} illustrates how infected individuals contribute to the specified vector abundance dynamic.
Note that, despite the absence of vectors during the winter period and the number of infectious hosts dropping to the order of $10^{-15}$, the infected individuals re-emerge. This biologically inconsistent observation underscores the challenges associated with modelling realistic inter-seasonal propagation of a pathogen. It justifies the introduction of the shrinking adjustment defined in Section \ref{subsec:shrinking-adj}.

The shrinking adjustment weakens transmission at low host prevalence, thus correcting any unrealistic inter-seasonal persistence of the epidemic. This result is established theoretically in Section~\ref{apx:proof_winter_extinction} of the Supplementary material, and illustrated numerically in Figures \ref{fig:output_sr} and \ref{fig:output_sr_V}.

\paragraph{Dependence of epidemic scale on entomological and epidemiological conditions.}

\begin{figure}[htb!]
\begin{subfigure}{.58\textwidth}
    \centering
    \includegraphics[width=.95\linewidth]{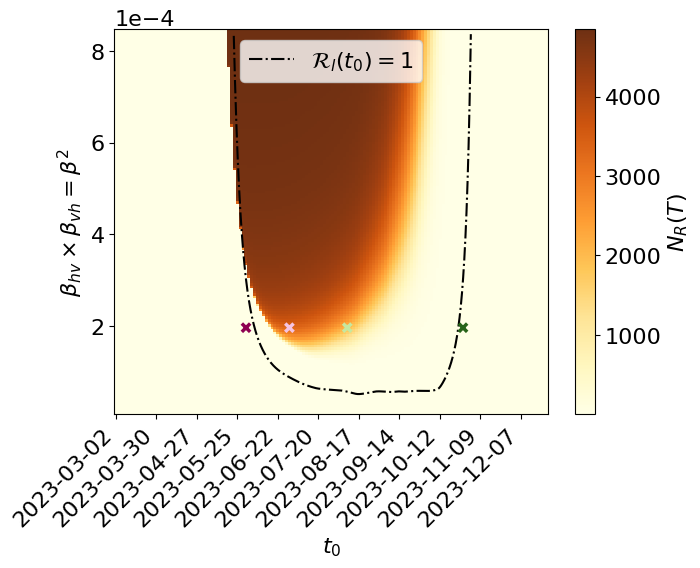}
    \caption{}
    \label{fig:heatmap}
\end{subfigure}\hfill
\begin{subfigure}{.41\textwidth}
    \centering   
    \vspace*{-1cm} \includegraphics[width=.99\linewidth]{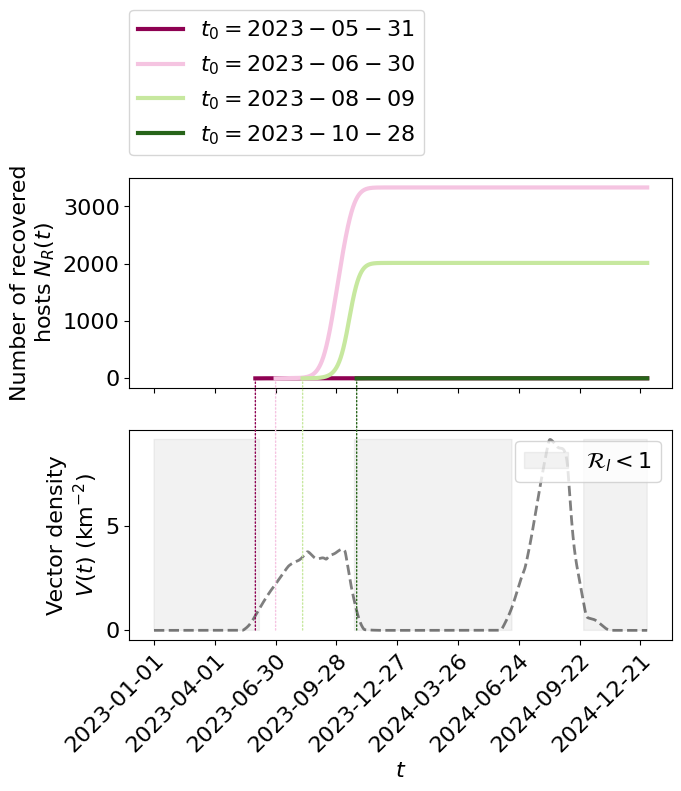}
    \vspace*{0.2cm}
    \caption{}
    \label{fig:t0_dependence}
\end{subfigure}
\caption{\small\textit{\textbf{(a) Total number of infected over the epidemic period in the host population $N_{R}(T)$} in terms of the transmission rates ($\beta_{hv}$ and $\beta_{vh}$) and of the date of introduction ($t_0$) with shrinking adjustment and with specified vector abundance.
The black dashed line corresponds to set of parameters where $\mathcal{R}_l = 1$. When $\mathcal{R}_l < 1$, conditions are not favourable for emergence and the epidemic fades out.
The four coloured crossed indicate the parameters used to generate the four simulations displayed in Figure \ref{fig:t0_dependence}.
$\gamma = 1/5$ days$^{-1}$, $\beta_{vh}=\beta_{hv}$. \\\textbf{(b) Evolution of the cumulated number of infected hosts} $N_R(t)$ for four dates of introduction at fixed transmission rates ($\beta_{hv} = \beta_{vh} = 1.32 \times 10^{-2}$). The corresponding time-varying vector abundance $V(t)$ is presented below.\\Specified vector abundance: entomological surveillance data of eggs of \Aa\ collected in Decize, France, in 2023 and 2024.}}
\label{fig}
\end{figure}

The epidemic outcome may vary depending on the entomological and epidemiological conditions at the time of introduction (e.g. depending on the number of susceptible vectors, the level of the infection pressure).
Numerical results show, under shrinking adjustment and in a favourable emergence context, that the total number of infected hosts increases with the transmission rates (see Figure \ref{fig:heatmap}).
Moreover, later introduction results in smaller epidemic scale (see Figures \ref{fig:heatmap} and \ref{fig:t0_dependence}).

These numerical results also highlight $\mathcal{R}_l(t_0)$ is an imperfect indicator of emergence. Indeed, as shown in Figure~\ref{fig:heatmap}, when the transmission parameter is large enough and for early pathogen introduction, an epidemic can grow even while $\mathcal{R}_l(t_0) < 1$. This occurs because $\mathcal{R}_l(t_0)$ is defined at a single point in time, without accounting for future variations in $V(t)$.

\paragraph{City-specific transmission context.} \label{subsec:city}
\begin{figure}[htb!]
        \centering
        \includegraphics[width=0.7\linewidth]{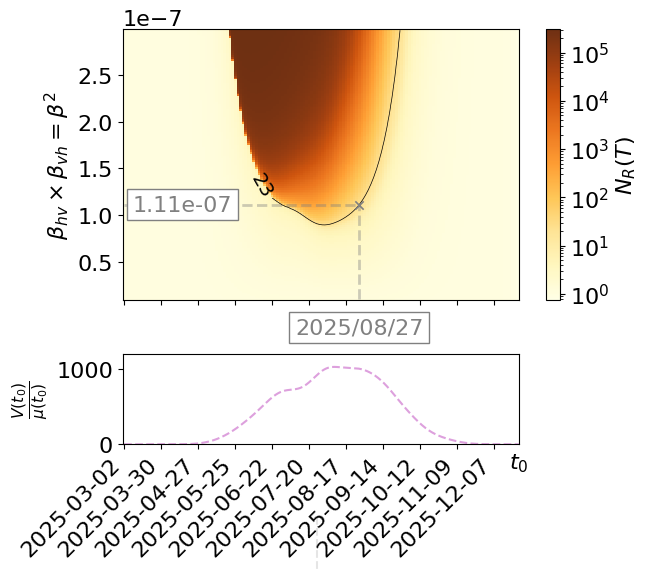}
    \caption{\small\textit{\textbf{Total number of infected over the epidemic period in the host population} $N_{R}(T)$ simulated in terms of the transmission rates ($\beta_{hv}$ and $\beta_{vh}$) and of the date of introduction ($t_0$), with shrinking adjustment, and \textbf{including observed indigenous cases of chikungunya virus disease in Nice, France} (the continuous black line corresponds to the isocline $N_R(T) = 23$, total number of reported cases during the Nice 2025 outbreak). The corresponding evolution of the $V/\mu$ ratio is displayed below. $\gamma = 1/5$ days$^{-1}$, $\beta_{vh}=\beta_{hv}=: \beta$.}}
    \label{fig:heatmap_Nice_2025}
\end{figure}

As the date of introduction and the number of indigenous cases are frequently examined by public health stakeholders, this heatmap also provides a means to estimate the transmission rate.
For instance, 23 cases of chikungunya virus disease were reported in Nice during the 2025 outbreak in mainland France. The clinical signs of the first case were reported on August 27$^{th}$, 2025 \cite{SpF_bulletin_2025}. Figure \ref{fig:heatmap_Nice_2025} then allows to graphically estimate the transmission rate corresponding to this outbreak ($\beta_{hv} \times \beta_{vh} = 1.11 \times 10^{-7}$).
Similarly, Figure S\ref{supfig:heatmaps_2025} of the Supplementary material allows to graphically estimate the transmission rate corresponding to the Bergerac and the Dijon outbreaks in 2025.
This method therefore enables characterization of the epidemiological context specific to each city and facilitates comparisons between locations (e.g. differences in personal protection measures).

\subsection{Results from prescribing the output of a complex population dynamics model}

\begin{figure}[htb!]
    \centering
    \includegraphics[width=0.7\linewidth]{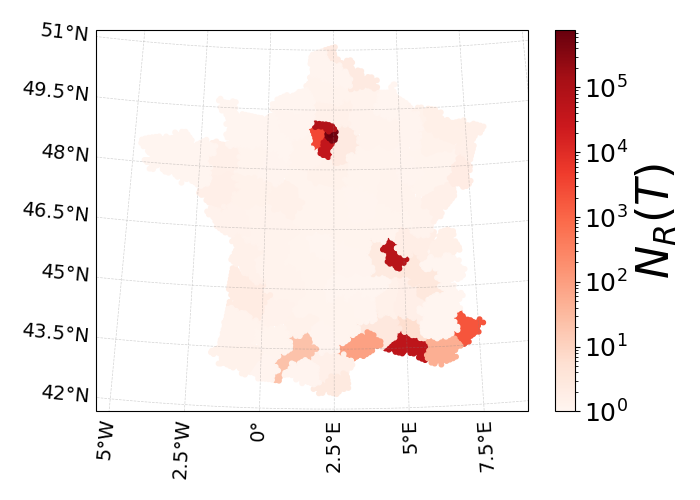}
    \caption{\small\textit{\textbf{Total number of infected hosts over the epidemic period in each French continental department}, with shrinking adjustment, based on DISTIGRI output, under the hypothesis of introduction of one infected host per department at $t_0$ : July 24$^{th}$, 2024. \\$T:$ November 2$^{nd}$, 2024, $\gamma = 1/5$, $\beta_{hv}=0.02$,
    $\beta_{vh}=0.008$}}.
    \label{fig:map}
\end{figure}

\paragraph{Department-level vulnerability to arbovirus transmission.}

By predicting \Aa\ distribution at an $8 \times 8$ km$^2$ resolution across France, the DISTIGRI model enables scaling up to broader geographic levels and facilitates comparisons of chikungunya vulnerability among larger administrative regions.
Here, DISTIGRI outputs are aggregated at the continental French administrative territory level and serve as inputs to our \Aa -human arbovirus model. Figure \ref{fig:map} illustrates the differences in epidemic outcome across departments. In addition, the \href{https://phoreas.biosp.inrae.fr}{PHOREAS web-app}\footnote{The PHOREAS web-app is available here: \url{https://phoreas.biosp.inrae.fr}} provides an interactive interface to visualise these results.
Both densities of susceptible hosts and susceptible vectors are key drivers of outbreak magnitude (see Figure S\ref{supfig:maps} in the Supplementary material). The larger the susceptible reservoir within a given territory, the greater its vulnerability to the emergence and spread of a vector-borne disease outbreak.
As spatial structure is not explicitly modelled here, this approach provides the mean epidemic behaviour and outcome within each territory, treated as independent subsystems. Future work will focus on incorporating spatial dynamics.


\newpage
\newpage
\section{Discussion}

We provide theoretical guarantees for the proper embedding of a specified vector population dynamic within a general vector-mediated transmission model.
The prescribed abundance $V(t)$ constrains the epidemiological state space through $V_S(t)+V_I(t)=V(t)$, and the compatibility conditions derived here ensure that this partition remains biologically admissible throughout the prescribed trajectory.
This framework is demonstrated through an application to an \Aa -borne disease.

Since prescribing $V(t)$ makes the model non-autonomous, we define the local reproduction index $\mathcal R_l(t_0)$ by freezing the system at the time of pathogen introduction. This index characterises whether the conditions prevailing at $t_0$ would favour local epidemic growth if they were to remain unchanged. The construction is conceptually related to instantaneous reproduction numbers, which characterise transmissibility assuming that current conditions remain unchanged \cite{cori_new_2013}. Although $\mathcal R_l(t_0)$ does not provide a necessary and sufficient criterion for pathogen emergence in the full non-autonomous system, it nevertheless offers a useful local characterisation of the transmission conditions at the time of introduction.

Over longer time scales, our model ensures the ability to capture effective propagation of the pathogen, including inter-annually. Strong seasonality of the vector population exposes a well-known limitation of continuous deterministic population models. When vector abundance becomes very low, infectious hosts may remain positive at sub-individual levels and seed transmission when vectors return. The shrinking adjustment proposed here attenuates host-to-vector transmission when the number of infectious hosts drops bellow one.
Hence, for the shrinking-adjusted system, we show that a sufficiently long period of sufficiently low vector abundance drives the infected compartments towards zero even if vector abundance subsequently increases. This provides a deterministic mechanism for inter-seasonal fade-out without imposing an arbitrary extinction event. The corresponding criteria are derived for the \Aa-human model assuming introduction through an infectious host; analogous arguments apply to vector-mediated introduction, and the underlying mechanism is not specific to this application.

While the primary objective of this work is to provide the key elements for handling a prescribed vector abundance and to unveil the dynamic of infected vectors, the application model was deliberately kept simple. Within the general framework presented in this article, the model could be extended to include additional features, such as direct transmission between hosts \cite{cousins_modelling_2019, ewing_novel_2021} or additional infectious stages \cite{chowell_assessing_2019, gao_prevention_2016, suparit_mathematical_2018, biswas_mathematical_2019}.
Moreover, the formulation of the infection pressure could be refined by explicitly accounting for the dependence of vector competence on its age and on temperature, the duration of its gonotrophic cycle, environmental conditions, and pathogen identity \cite{mordecai_detecting_2017, luhken_high_2024, zouache_three-way_2014, yeh_anti-immune_2023, viglietta_vector_2021}. Such an extension would provide a more biologically realistic description of pathogen transmission.
Similarly, regarding adult-vector mortality, the rate $\mu(t)$ was specified as a function of daily mean temperature. A more refined approach would incorporate temperature fluctuations, which have been shown to enhance mosquito survival \cite{jian_effects_2023}. On the other hand, under the demographic interpretation adopted here, variations in total vector abundance are attributed to recruitment and adult mortality, while infection only redistributes vectors between epidemiological compartments.
For increased accuracy, one could distinguish between birth and mortality rates within the vector population with several independent observations of a corresponding stochastic trajectory \cite{bansaye_stochastic_2017}. 

The \Aa-borne chikungunya application shows how the temporal structure of vector abundance shapes epidemic outcomes. The timing of pathogen introduction relative to the seasonal vector trajectory strongly affects epidemic size even when transmission parameters remain unchanged.
Although only the dependence on the introduction date is illustrated, any other effect that reduces the reservoir of susceptible vectors likewise diminishes the epidemic extent. By directly affecting the vector abundance, global rise of temperatures or management interventions, including vector-control or personal protection control strategies, may substantially influence both pathogen emergence risk and epidemic scale \cite{tjaden_mosquito-borne_2018, ewing_novel_2021, caminade_suitability_2012, garrido_zornoza_effect_2024, suparit_mathematical_2018, tegar_surveillance_2026, demers_dynamic_2018}.
This framework can therefore be used prospectively with vector trajectories generated under environmental or management scenarios to inform public decision-making.

Finally, while they provide valuable qualitative insights, the numerical results should be interpreted with caution from a quantitative perspective.
Regarding entomological surveillance data, although ovitraps are inexpensive, non-intrusive and can be operated efficiently with minimal training \cite{gubler_surveillance_2014, focks_review_2003}, eggs counting is non-exhaustive. Ovitrap egg counts provide useful information on seasonal mosquito activity but do not directly measure the abundance of adult females involved in transmission. Laboratory and field data are required to better estimate the number of eggs laid, and then infer the number of susceptible females actively contributing to the epidemiological dynamics.
Nonetheless, our work demonstrates that entomological surveillance strategies generate data that, despite their potentially low resolution, are highly valuable within this research context.

Overall, this conditional modelling strategy extends existing approaches by combining entomological surveillance or specialised population models with mechanistic transmission models.
Its adaptability to other vector-borne pathosystems expands its potential applications, whether it be in animal, human or plant health.

\newpage
\paragraph{Ethics.} This work did not require ethical approval from a human subject or animal welfare committee.

\paragraph{Data accessibility.} All Python code developed to run simulations and generate figures is publicly available at \Lil [link to come]\Bk. Vector-abundance data obtained from DISTIGRI, the external detailed population dynamic model, are available at the same link. In addition, the \href{https://phoreas.biosp.inrae.fr}{PHOREAS web-app} provides an interactive interface to visualise results obtained by prescribing vector abundance data from DISTIGRI.
Entomological surveillance data are available on demand to the French \textit{Agences Régionales de Santé} and their operators (\textit{Altopictus}, \textit{Entente Interdépartementale pour la Démoustication Méditerranée}, \textit{Entente Interdépartementale Rhône-Alpes pour la Démoustication}).

\paragraph{Authors' Contributions.} L.I., J.P., L.R. and S.S. conceived the model, discussed the results, and wrote the final manuscript. O.B. adapted the implementation of the DISTIGRI model to meet the specific requirements of this study. L.I. carried out the numerical simulations and generated the associated figures, and established agreements for sharing entomological data with French public health agencies (\textit{Agences Régionales de Santé}) and their operators. 

\paragraph{Declaration of AI use} We used AI-assisted technologies to enhance the manuscript's clarity and grammatical correctness. The tool was employed under the direct supervision of the authors and did not influence the study's design, interpretation, or conclusions.

\paragraph{Conflict of interest.} The authors declare no conflict of interest.

\paragraph{Funding.} This work is funded by the French Ministry of Agriculture. In addition, it received support from the Maths-ArboV project which is funded by the PEPR Maths-Vives managed by the National Research Agency (ANR), under the France 2030 program, reference ANR-23-EXMA-0001.

\paragraph{Ackowledgments.} The authors gratefully acknowledge Frédéric Simard for helpful discussions, together with the French \textit{Agences Régionales de Santé} and their operators (\textit{Altopictus}, \textit{Entente Interdépartementale pour la Démoustication Méditerranée}, \textit{Entente Interdépartementale Rhône-Alpes pour la Démoustication}), for sharing their datasets.


\newpage
\printbibliography


\clearpage
\begin{appendices}


\section{Processes affecting the infectious-vector compartment and local
reproduction index}
\label{apx:r(t)}

\paragraph{Processes affecting the infectious-vector compartment.}
The term $r(t)V_I(t)$ represents the net contribution to the
infectious-vector compartment of processes other than host-to-vector
transmission. It must be specified consistently with the prescribed total
vector abundance $V(t)$. This choice determines both the biological interpretation
of the model and the compatibility condition required to ensure
\begin{equation}
V_S(t)+V_I(t)=V(t),
\qquad
V_S(t)\geq0,
\qquad
V_I(t)\geq0.
\end{equation}
The corresponding admissibility conditions are derived in
Section~\ref{apx:critera_A4_A5} of the Supplementary material.

We consider the vector subsystem
\begin{equation}
\left\{
\begin{aligned}
V_S'(t)
&=
V'(t)-V_I'(t),
\\
V_I'(t)
&=
\lambda_{HV}(H_I(t))
\left(
V(t)-V_I(t)
\right)
+
r(t)V_I(t).
\end{aligned}
\right.
\label{eq:apx_general_vector_subsystem}
\end{equation}
We now provide demographic interpretations of $r(t)$.

Suppose first that the total vector population can be described through a
per capita birth rate $b(t)\geq0$ and a per capita mortality rate $\mu(t)>0$.
On any interval on which $V(t)>0$,
\begin{equation}
V'(t)
=
\left(
b(t)-\mu(t)
\right)V(t),
\end{equation}
or equivalently
\begin{equation}
\frac{V'(t)}{V(t)}
=
b(t)-\mu(t).
\label{eq:bmu}
\end{equation}
Thus, the prescribed abundance $V(t)$ determines only the net demographic
growth rate $b(t)-\mu(t)$ and does not determine $b(t)$ and $\mu(t)$
separately. At least one of these two rates must therefore be specified
independently if both are required by the infection model.

In the absence of vertical transmission, and assuming that infectious vectors
remain infectious throughout their lifespan and experience the same mortality
rate as susceptible vectors, mortality is the only process contributing to
$r(t)$. We therefore set
$$
r(t)=-\mu(t).
$$
This is the formulation used in the main application.

More generally, let $\nu(t)\geq0$ denote a possible per-capita rate at which infectious
vectors lose their infectious status and return to the susceptible-vector
compartment. If vertical transmission is included, let
$\alpha\in[0,1]$ denote the fraction of offspring produced by infectious
vectors that enter the infectious-vector class. Then
\begin{equation}
r(t)
=
\alpha b(t)-\mu(t)-\nu(t),
\label{eq:apx_r_vertical}
\end{equation}
and the infectious vector equation reads
\begin{equation}
V_I'(t)
=
\lambda_{HV}(H_I(t))
\left(
V(t)-V_I(t)
\right)
+
\alpha b(t)V_I(t)
-
\mu(t)V_I(t)-
\nu(t)V_I(t).
\end{equation}
Since
\begin{align}
V'(t)-r(t)V(t)
&=
\left(
b(t)-\mu(t)
\right)V(t)
-
\left(
\alpha b(t)-\mu(t)-\nu(t)
\right)V(t)
\nonumber\\
&=
\left(
(1-\alpha)b(t)+\nu(t)
\right)V(t),
\end{align}
the compatibility condition introduced in
Section~\ref{apx:critera_A4_A5} of the Supplementary material is then
\begin{equation}
\left(
(1-\alpha)b(t)+\nu(t)
\right)V(t)
\geq0.
\label{eq:apx_vertical_compatibility}
\end{equation}
Hence the construction is automatically compatible with
$V_I(t)\leq V(t)$ whenever $b(t)\geq0$, $\nu(t)\ge 0$ and $\alpha\in[0,1]$.

Under the assumptions of the main application, vertical transmission is
absent and infectious vectors remain infectious throughout their lifespan.
Hence,
\begin{equation}
V'(t)+\mu(t)V(t)
=
b(t)V(t),
\end{equation}
and equation~\eqref{eq:apx_vertical_compatibility} reduces to
\begin{equation}
b(t)V(t)
\geq0.
\label{eq:apx_novertical_compatibility}
\end{equation}
This is the compatibility condition used for the prescribed vector abundance
in the main application.

\paragraph{Local reproduction index of the frozen linear model.}

We now compute the local reproduction index of the epidemiological
environment frozen at the introduction time $t_0$, for the linear infection
pressures
\begin{equation}
\lambda_{VH}(V_I)
=
\beta_{vh}V_I,
\qquad
\lambda_{HV}(H_I)
=
\beta_{hv}H_I.
\end{equation}

Set
\begin{equation}
V_0:=V(t_0),
\qquad
b_0:=b(t_0),
\qquad
\mu_0:=\mu(t_0),
\qquad
\nu_0:=\nu(t_0).
\end{equation}
The corresponding disease-free state is
\begin{equation}
(H_S,H_I,H_R,V_S,V_I)
=
(H_{S,0},0,H_{R,0},V_0,0).
\end{equation}

Linearising the infected subsystem around this state gives
\begin{equation}
\begin{pmatrix}
H_I\\
V_I
\end{pmatrix}'
=
\begin{pmatrix}
-\gamma
&
\beta_{vh}H_{S,0}
\\
\beta_{hv}V_0
&
\alpha b_0-\mu_0-\nu_0
\end{pmatrix}
\begin{pmatrix}
H_I\\
V_I
\end{pmatrix}.
\label{eq:apx_linear_infected_vertical}
\end{equation}

The next-generation decomposition separates new infection terms from
transition and removal terms. Let
\begin{equation}
\mathsf F
=
\begin{pmatrix}
0
&
\beta_{vh}H_{S,0}
\\
\beta_{hv}V_0
&
\alpha b_0
\end{pmatrix},
\qquad
\mathsf T
=
\begin{pmatrix}
\gamma & 0\\
0 & \mu_0 + \nu_0
\end{pmatrix}.
\end{equation}
The next-generation matrix is
\begin{equation}
\mathsf K
=
\mathsf F\mathsf T^{-1}
=
\begin{pmatrix}
0
&
\dfrac{\beta_{vh}H_{S,0}}{\mu_0 + \nu_0}
\\[0.9em]
\dfrac{\beta_{hv}V_0}{\gamma}
&
\dfrac{\alpha b_0}{\mu_0 + \nu_0}
\end{pmatrix}.
\label{eq:apx_NGM_vertical}
\end{equation}

Define
\begin{equation}
\mathcal R_{l,VH}(t_0)
:=
\frac{\beta_{vh}H_{S,0}}{\mu_0 + \nu_0},
\qquad
\mathcal R_{l,HV}(t_0)
:=
\frac{\beta_{hv}V_0}{\gamma},
\end{equation}
and
\begin{equation}
\mathcal R_{l,VV}(t_0)
:=
\frac{\alpha b_0}{\mu_0 + \nu_0}.
\end{equation}
Then
\begin{equation}
\mathsf K
=
\begin{pmatrix}
0
&
\mathcal R_{l,VH}(t_0)
\\
\mathcal R_{l,HV}(t_0)
&
\mathcal R_{l,VV}(t_0)
\end{pmatrix}.
\end{equation}

The local reproduction index of the frozen system is the spectral radius of
$\mathsf K$:
\begin{equation}
\mathcal R_l(t_0)
=
\frac{
\mathcal R_{l,VV}(t_0)
+
\sqrt{
\mathcal R_{l,VV}(t_0)^2
+
4
\mathcal R_{l,VH}(t_0)
\mathcal R_{l,HV}(t_0)
}
}
{2}.
\label{eq:apx_Rl_vertical}
\end{equation}

Under the assumptions in the main text, vertical transmission is absent ($\alpha=0$), infectious vectors remain infectious throughout their lifespan ($\nu_0 = 0)$ and experience the same mortality rate as susceptible vectors, therefore

\begin{equation}
\mathcal R_{l,VV}(t_0)=0.
\end{equation}
and equation~\eqref{eq:apx_Rl_vertical} then reduces to
\begin{equation}
\mathcal R_l(t_0)
=
\sqrt{
\mathcal R_{l,VH}(t_0)
\mathcal R_{l,HV}(t_0)
}
=
\sqrt{
H_{S,0}V_0
\frac{
\beta_{hv}\beta_{vh}
}{
\gamma\mu_0
}
}.
\label{eq:apx_Rl_no_vertical}
\end{equation}

\end{appendices}

\clearpage
\renewcommand{\theequation}{S\arabic{equation}}
\renewcommand\thefigure{S.\arabic{figure}}   
\renewcommand{\thesection}{S.\arabic{section}}
\renewcommand{\thesubsection}{S.\arabic{section}.\arabic{subsection}}
\setcounter{figure}{0}
\setcounter{section}{0} 
\setcounter{equation}{0} 
\makeatletter
\@removefromreset{equation}{section}
\makeatother

\section*{Modelling vector-borne disease dynamics with externally prescribed vector abundance\\---------\\Supplementary material}


\section{Admissibility conditions for the infectious vector compartment}
\label{apx:critera_A4_A5}

We derive sufficient conditions ensuring that the infectious vector
compartment remains biologically admissible,
\begin{equation}
0\leq V_I(t)\leq V(t),
\end{equation}
and therefore that
\begin{equation}
V_S(t)=V(t)-V_I(t)\geq0.
\end{equation}

Consider
\begin{equation}
V_I'(t)
=
\Lambda(t)
\left(
V(t)-V_I(t)
\right)
+
r(t)V_I(t),
\label{eq:apx_VI_generic}
\end{equation}
where
\begin{equation}
\Lambda(t)
:=
\lambda_{HV}\left(\mathbf H_I(t)\right).
\end{equation}
Assume that $V$, $\Lambda$, and $r$ are continuous on $[t_0,T]$, with
$V(t)\geq0$, and let
\begin{equation}
0\leq V_I(t_0)\leq V(t_0).
\end{equation}

\paragraph{Non-negativity of $V_I$:
criterion~\ref{item:lambda_HV_positive}.}

Equation~\eqref{eq:apx_VI_generic} can be written as
\begin{equation}
V_I'(t)
+
\left(
\Lambda(t)-r(t)
\right)V_I(t)
=
\Lambda(t)V(t).
\end{equation}
Variation of constants gives
\begin{align}
V_I(t)
&=
V_I(t_0)
\exp
\left(
-\int_{t_0}^{t}
\left(
\Lambda(u)-r(u)
\right)\,du
\right)
\nonumber\\
&\quad
+
\int_{t_0}^{t}
\Lambda(s)V(s)
\exp
\left(
-\int_s^t
\left(
\Lambda(u)-r(u)
\right)\,du
\right)
\,ds.
\label{eq:apx_VI_variation_general}
\end{align}

Hence
\begin{equation}
\Lambda(t)\geq0
\qquad
\text{for all }t\in[t_0,T]
\label{eq:apx_lambda_positive_condition}
\end{equation}
is sufficient to ensure
\begin{equation}
V_I(t)\geq0
\qquad
\text{for all }t\in[t_0,T].
\end{equation}

\paragraph{Upper bound $V_I\leq V$:
criterion~\ref{item:V_r_compatible}.}

Let
\begin{equation}
Y(t):=V(t)-V_I(t)=V_S(t).
\end{equation}
Using equation~\eqref{eq:apx_VI_generic},
\begin{align}
Y'(t)
&=
V'(t)-V_I'(t)
\nonumber\\
&=
V'(t)-r(t)V(t)
+
\left(
r(t)-\Lambda(t)
\right)Y(t).
\end{align}
Hence
\begin{equation}
Y'(t)
+
\left(
\Lambda(t)-r(t)
\right)Y(t)
=
V'(t)-r(t)V(t).
\end{equation}
Variation of constants gives
\begin{align}
Y(t)
&=
Y(t_0)
\exp
\left(
-\int_{t_0}^{t}
\left(
\Lambda(u)-r(u)
\right)\,du
\right)
\nonumber\\
&\quad
+
\int_{t_0}^{t}
\left(
V'(s)-r(s)V(s)
\right)
\exp
\left(
-\int_s^t
\left(
\Lambda(u)-r(u)
\right)\,du
\right)
\,ds.
\label{eq:apx_VS_variation_general}
\end{align}

Since $Y(t_0)\geq0$, the condition
\begin{equation}
V'(t)-r(t)V(t)\geq0
\qquad
\text{for all }t\in[t_0,T]
\label{eq:apx_Vr_compatibility}
\end{equation}
is sufficient to ensure
\begin{equation}
Y(t)\geq0,
\end{equation}
and therefore
\begin{equation}
V_I(t)\leq V(t)
\qquad
\text{for all }t\in[t_0,T].
\end{equation}

Combining
equations~\eqref{eq:apx_lambda_positive_condition}
and~\eqref{eq:apx_Vr_compatibility}, we obtain the sufficient admissibility
conditions
\begin{equation}
0\leq V_I(t)\leq V(t)
\qquad
\text{for all }t\in[t_0,T].
\end{equation}

In the main application, where vertical transmission and other removals from
the infectious class are absent, $r(t)=-\mu(t)$, and the compatibility
condition becomes
\begin{equation}
V'(t)+\mu(t)V(t)\geq0.
\label{eq:apx_main_compatibility}
\end{equation}

\clearpage

\section{Regularisation of prescribed vector abundance}
\label{apx:data_reg}

When the prescribed vector abundance is obtained from observations or from an
external model, it may not satisfy the compatibility condition
\begin{equation}
V'(t)-r(t)V(t)\geq0
\qquad
\text{for all }t\in[t_0,T]
\end{equation}
required by Section~\ref{apx:critera_A4_A5}. We therefore introduce a
smoothing procedure designed to produce a regularised trajectory compatible
with the specified rate $r(t)$.

Let $V\in L^1([t_0,T])$ be non-negative and not identically zero. For
$u>0$, define the Gaussian kernel
\begin{equation}
G_u(x)
=
\frac{1}{\sqrt{2\pi}\,u}
\exp
\left(
-\frac{x^2}{2u^2}
\right)
\end{equation}
and the boundary-corrected Gaussian smoother
\begin{equation}
V_u(t)
=
\frac{
\displaystyle
\int_{t_0}^{T}
V(s)G_u(t-s)\,ds
}{
\displaystyle
\int_{t_0}^{T}
G_u(t-s)\,ds
},
\qquad
t\in[t_0,T].
\label{eq:apx_Vu}
\end{equation}

The denominator corrects for truncation of the kernel near the boundaries of
the observation interval. In particular, the operator preserves constant
trajectories. For every $u>0$, $V_u$ is smooth on $[t_0,T]$.

For a prescribed minimum smoothing scale $u_{\min}>0$, define the set of
admissible smoothing scales
\begin{equation}
\mathcal U
:=
\left\{
u\geq u_{\min}:
V_u'(t)-r(t)V_u(t)\geq0
\text{ for every }t\in[t_0,T]
\right\}.
\label{eq:apx_U_regularisation}
\end{equation}

Whenever $\mathcal U$ is non-empty, we define
\begin{equation}
\sigma^\star
:=
\min\mathcal U.
\label{eq:apx_sigma_star}
\end{equation}
The regularised prescribed abundance is then
\begin{equation}
V_{\mathrm{reg}}(t)
:=
V_{\sigma^\star}(t).
\end{equation}

The minimum in equation~\eqref{eq:apx_sigma_star} is well defined whenever
$\mathcal U$ is non-empty. Indeed, for $u>0$ the mapping
\begin{equation}
u
\mapsto
\min_{t\in[t_0,T]}
\left(
V_u'(t)-r(t)V_u(t)
\right)
\end{equation}
is continuous, so that $\mathcal U$ is closed.

In the main application,
\begin{equation}
r(t)=-\mu(t),
\qquad
\mu(t)\geq\underline\mu>0.
\end{equation}
For large $u$, $V_u$ converges uniformly to the temporal average of $V$ and
$V_u'$ converges uniformly to zero. Since $V$ is non-negative and not
identically zero, its temporal average is strictly positive. Consequently,
\begin{equation}
V_u'(t)+\mu(t)V_u(t)>0
\end{equation}
for sufficiently large $u$, and $\mathcal U$ is therefore non-empty in the
setting used in the main text.

In the numerical implementation, we impose
\begin{equation}
u_{\min}=1~\text{day}
\end{equation}
to avoid smoothing at sub-daily time scales. Figure
\ref{fig:apx:V_regularisation} illustrates the resulting regularisation.

\begin{figure}[h!]
    \centering
    \includegraphics[width=0.7\linewidth, clip, trim={0 0 0 34}]{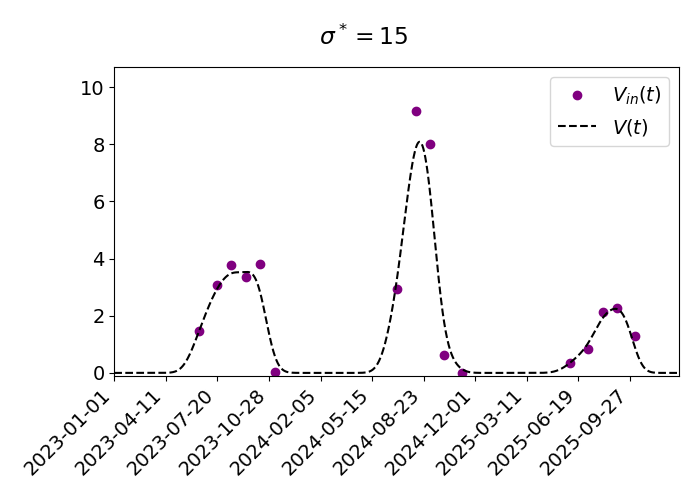}
    \caption{\textit{
    Vector-abundance proxy derived from \Aa\ ovitrap data collected in
    Decize, France, between 2023 and 2025 (dots), together with the
    regularised trajectory obtained using the smallest admissible Gaussian
    smoothing scale $\sigma^\star = 15$ (line).
    }}
    \label{fig:apx:V_regularisation}
\end{figure}


\clearpage

\section{Temperature-dependent survival of \Aa}
\label{apx:T}

\subsection{Adult female temperature-dependent mortality rate}
\label{apx:T_mu_adult}

\Aa\ is an ectothermic species, and temperature affects several components of
its life cycle, including development and adult survival.

In the present model, the vector mortality rate $\mu$ is temperature dependent
and is derived from the female \Aa\ longevity data reported by Brady et al.
\cite{brady_modelling_2013,brady_data_2013}. The 50\% survival curve is used
to parameterise the temperature dependence of adult female longevity in the
model.

To prevent unrealistically large removal rates when the fitted longevity
becomes very small, adult female longevity is bounded below by one hour in
the numerical implementation.
Besides, $\mu$ is smoothed to prevent any discontinuity and, thus, to make sure the admissibility condition \ref{item:r_continuous} is satisfied.
The corresponding relationship between
temperature and vector longevity is illustrated in
Figure~\ref{fig:apx:T}.

Hourly temperature data on an $8\times8$~km$^2$ grid covering France were
obtained from the SAFRAN reanalysis database
\cite{durand_reanalysis_2009}.

\begin{figure}[h]
    \centering
    \includegraphics[width=0.9\linewidth]{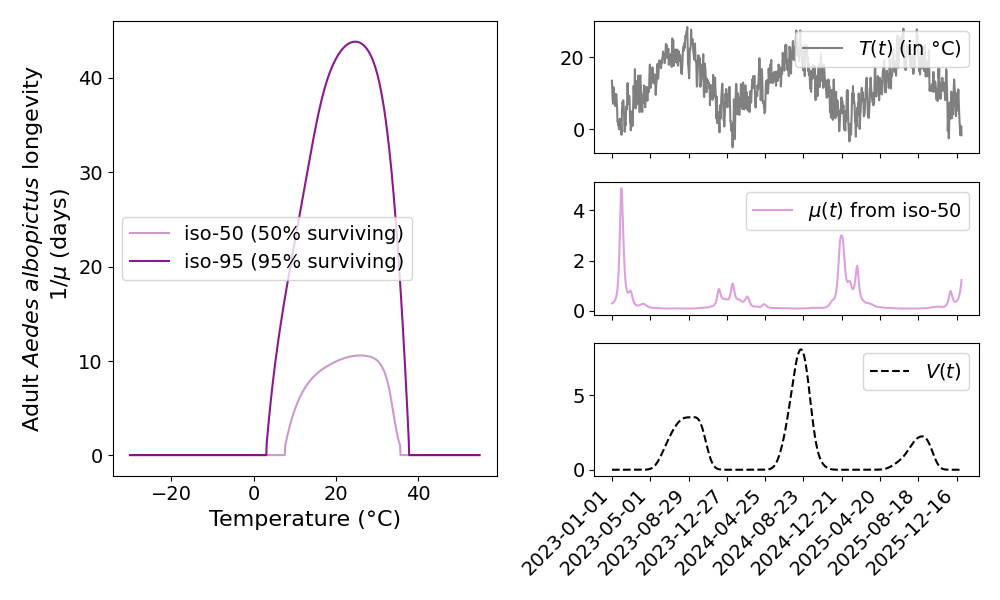}
    \caption{\textit{
    \textbf{Left:} Female \Aa\ longevity as a function of temperature under
    field conditions. The curves correspond to 50\% and 95\% survival
    probabilities. Data source: Brady et al. (2013)
    \cite{brady_data_2013}.
    \textbf{Right:} Mean daily temperature $T(t)$, associated mortality rate $\mu(t)$, and prescribed vector abundance
    over time in Decize, France.
    }}
    \label{fig:apx:T}
\end{figure}

\subsection{Larval temperature-dependent survival rate}
\label{apx:larval_survival}

The temperature-dependent survival rate from egg hatching to adult emergence
in \Aa\ was experimentally measured by Delatte et al.
\cite{delatte_influence_2009}. The experimental data and the fitted
relationship used in the present study are shown in
Figure~\ref{fig:larval_survival}.

\begin{figure}[h]
    \centering
    \includegraphics[width=0.7\linewidth]{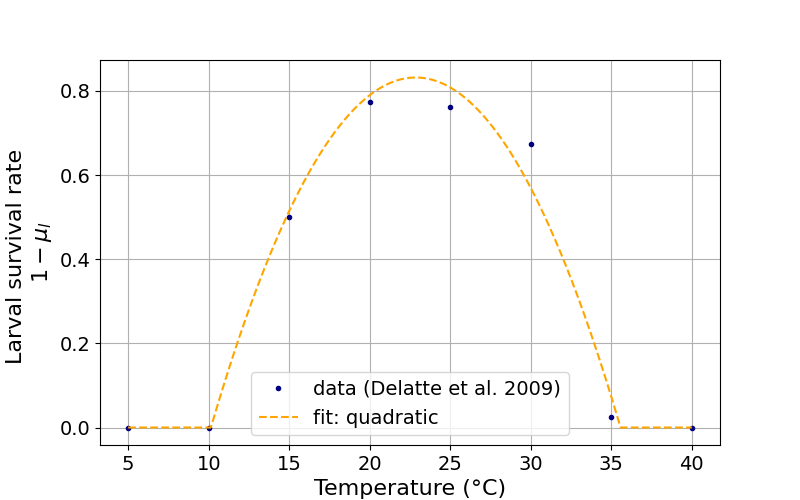}
    \caption{\textit{
    \textbf{Survival rate from hatching to adult emergence of \Aa} measured by
    Delatte et al. at eight constant temperatures
    \cite{delatte_influence_2009}.
    }}
    \label{fig:larval_survival}
\end{figure}


\clearpage

\section{Well-posedness of the \Aa-human arbovirus model}
\label{apx:well-pos-example}

We verify here that the \Aa-human arbovirus model introduced in
Section~\ref{sec:Aa-human-model} satisfies the regularity and admissibility
conditions used in the general modelling framework.

Throughout this section, the biological state space is restricted to
non-negative host compartments satisfying
\begin{equation}
H_S(t)+H_I(t)+H_R(t)=H_0<+\infty
\end{equation}
and to vector compartments satisfying
\begin{equation}
0\leq V_I(t)\leq V(t).
\end{equation}

\paragraph{Host dynamics:
criterion~\ref{item:F_H_Lz}.}

The host field is
\begin{equation}
\mathbf F_H(H_S,H_I,H_R,V_I)
=
\begin{pmatrix}
-\beta_{vh}V_IH_S
\\
\beta_{vh}V_IH_S-\gamma H_I
\\
\gamma H_I
\end{pmatrix}.
\label{eq:apx_host_field}
\end{equation}
It is polynomial in its state variables and is therefore locally Lipschitz on
$\mathbb R^4$. In particular, it is Lipschitz on every bounded subset of the
biological state space.

\paragraph{Host-to-vector infection pressure:
criterion~\ref{item:lambda_HV_Lz}.}

For the standard linear model,
\begin{equation}
\lambda_{HV}(H_I)
=
\beta_{hv}H_I.
\end{equation}
This function is globally Lipschitz. On the biological host simplex,
\begin{equation}
0\leq H_I(t)\leq H_0,
\end{equation}
and therefore
\begin{equation}
0\leq
\lambda_{HV}(H_I(t))
\leq
\beta_{hv}H_0.
\end{equation}
Thus, although the linear map is not bounded on the whole real line, it is
uniformly bounded along all biologically admissible trajectories.

For the shrinking-adjusted model introduced in
Section~\ref{apx:emgc_ext_under_shrinking},
\begin{equation}
\lambda_{HV}(H_I)
=
\beta_{hv}g(H_I),
\end{equation}
where $g$ is locally Lipschitz and satisfies
\begin{equation}
0\leq g(H_I)\leq H_I
\qquad
\text{for }H_I\geq0.
\end{equation}
The same boundedness property therefore holds on the biological host simplex.

\paragraph{Continuity of the rate $r(t)$:
criterion~\ref{item:r_continuous}.}

In the main application,
\begin{equation}
r(t)
=
-\mu(T(t)).
\end{equation}
If the temperature trajectory $T(t)$ is continuous and the fitted mortality
function $\mu(T)$ is continuous, as assumed in
Section~\ref{apx:T_mu_adult}, then
\begin{equation}
t\mapsto r(t)
\end{equation}
is continuous.

\paragraph{Non-negativity of the infection pressure:
criterion~\ref{item:lambda_HV_positive}.}

On the biological state space, $H_I\geq0$, and therefore the standard
infection pressure
\begin{equation}
\lambda_{HV}(H_I)=\beta_{hv}H_I
\end{equation}
is non-negative.

If an extension to the whole real line is required for the application of a
general existence theorem, the linear pressure may be extended as
\begin{equation}
\lambda_{HV}^{\mathrm{ext}}(H_I)
=
\beta_{hv}\max\{H_I,0\}.
\end{equation}
This extension is globally Lipschitz and non-negative and coincides with the
original infection pressure for $H_I\geq0$.

The shrinking-adjusted function $g$ introduced in
Section~\ref{apx:emgc_ext_under_shrinking} is already extended by
\begin{equation}
g(H_I)=0
\qquad
\text{for }H_I\leq0,
\end{equation}
and is therefore non-negative on $\mathbb R$.

\paragraph{Compatibility between prescribed abundance and its epidemiological
partition:
criterion~\ref{item:V_r_compatible}.}

Since
\begin{equation}
r(t)=-\mu(t),
\end{equation}
the compatibility condition from
Section~\ref{apx:critera_A4_A5} becomes
\begin{equation}
V'(t)+\mu(t)V(t)\geq0.
\label{eq:apx_application_compatibility}
\end{equation}
When the raw prescribed abundance does not satisfy
equation~\eqref{eq:apx_application_compatibility}, the regularisation
procedure of Section~\ref{apx:data_reg} is applied.

Consequently, under the regularity assumptions stated above and the
compatibility condition
\eqref{eq:apx_application_compatibility}, the \Aa-human model remains in the
biologically admissible region
\begin{equation}
H_S,H_I,H_R\geq0,
\qquad
0\leq V_I\leq V,
\qquad
V_S=V-V_I\geq0.
\end{equation}


\clearpage

\section{Extinction and finite-amplitude epidemic growth under shrinking adjustment}
\label{apx:emgc_ext_under_shrinking}

A classical limitation of deterministic epidemic models arises when infected
populations become very small. In the present setting, seasonal decreases in
vector abundance may drive the infected host and vector compartments towards
arbitrarily small but strictly positive values. Such residual densities may
subsequently seed renewed transmission when vector abundance increases again,
leading to an artificial inter-seasonal re-emergence.

To mitigate this deterministic artefact, we introduced, in
equation~\eqref{eq:lambdahv_atto}, the shrinking adjustment
\begin{equation}
\lambda_{HV}(H_I)
=
\beta_{hv}g(H_I),
\qquad
g(H_I)
=
\begin{cases}
0,
& \text{if }H_I\leq0,
\\[4pt]
\dfrac{H_I^2}{H_\star},
& \text{if }0<H_I<H_\star,
\\[6pt]
H_I,
& \text{if }H_I\geq H_\star.
\end{cases}
\label{eq:apx_shrinking}
\end{equation}
The extension $g(H_I)=0$ for $H_I<0$ does not modify the dynamics in the
biologically relevant domain $H_I\geq0$, but makes $g$ locally Lipschitz on
$\mathbb R$ and facilitates the invariance argument below.

In the application considered in the main text,
\begin{equation}
H_\star=\frac{1}{|\Omega|},
\end{equation}
so that the adjustment acts when the infectious-host density corresponds to
fewer than one infectious individual in the spatial domain $\Omega$.

The objective of this supplementary section is twofold. First, we establish a sufficient
condition ensuring extinction of the infectious compartments in a frozen
epidemiological environment. Second, we show that the shrinking adjustment
does not prevent epidemic growth from a finite-amplitude introduction.

These results are derived for the \Aa-human arbovirus model introduced in
Section~\ref{sec:Aa-human-model}. Since the full model is non-autonomous, the
analysis in this supplementary section is conducted in the epidemiological environment
frozen at the time of pathogen introduction $t_0$.

\subsection*{Frozen admissible system}

Set
\begin{equation}
V_0:=V(t_0),
\qquad
\mu_0:=\mu(t_0),
\end{equation}
and assume
\begin{equation}
V_0>0,
\qquad
\mu_0>0.
\end{equation}

Freezing the prescribed vector abundance means imposing
\begin{equation}
V(t)=V_0
\qquad
\text{for all }t\geq t_0.
\end{equation}
Population conservation therefore gives
\begin{equation}
V_S(t)=V_0-V_I(t).
\label{eq:apx_VS_conservation}
\end{equation}

The resulting frozen system is
\begin{equation}
\left\{
\begin{aligned}
H_S'(t)
&=
-\beta_{vh}V_I(t)H_S(t),
\\
H_I'(t)
&=
\beta_{vh}V_I(t)H_S(t)-\gamma H_I(t),
\\
H_R'(t)
&=
\gamma H_I(t),
\\
V_I'(t)
&=
\beta_{hv}g(H_I(t))
\left(
V_0-V_I(t)
\right)
-\mu_0V_I(t),
\\
V_S(t)
&=
V_0-V_I(t).
\end{aligned}
\right.
\label{eq:apx_frozen_shrinking}
\end{equation}

Let
\begin{equation}
H_0
:=
H_S(t_0)+H_I(t_0)+H_R(t_0).
\end{equation}
The biologically admissible state space for the reduced system is
\begin{equation}
\mathcal D_{V_0}
=
\left\{
(H_S,H_I,H_R,V_I)\in\mathbb R_+^4:
H_S+H_I+H_R=H_0,\;
0\leq V_I\leq V_0
\right\}.
\end{equation}

The set $\mathcal D_{V_0}$ is closed and convex. We verify positive
invariance using the tangent-cone criterion.

First,
\begin{equation}
H_S'+H_I'+H_R'=0,
\end{equation}
so the vector field is tangent to the hyperplane
\begin{equation}
H_S+H_I+H_R=H_0.
\end{equation}

For each active non-negativity constraint,
\begin{align}
H_S'&=0
&&\text{when }H_S=0,
\\
H_I'&=\beta_{vh}V_IH_S\geq0
&&\text{when }H_I=0,
\\
H_R'&=\gamma H_I\geq0
&&\text{when }H_R=0.
\end{align}
At the boundaries of the vector interval,
\begin{equation}
V_I'
=
\beta_{hv}g(H_I)V_0
\geq0
\qquad
\text{when }V_I=0,
\end{equation}
whereas
\begin{equation}
V_I'
=
-\mu_0V_0
<0
\qquad
\text{when }V_I=V_0.
\end{equation}

These conditions hold simultaneously when several constraints are active.
Therefore, at every boundary point of $\mathcal D_{V_0}$, the vector field
belongs to the tangent cone of $\mathcal D_{V_0}$. Since
$\mathcal D_{V_0}$ is closed and convex and the vector field in
equation~\eqref{eq:apx_frozen_shrinking} is locally Lipschitz, the Nagumo
tangent-cone invariance criterion implies that $\mathcal D_{V_0}$ is
positively invariant.

Consequently,
\begin{equation}
H_S(t),H_I(t),H_R(t)\geq0,
\qquad
0\leq V_I(t)\leq V_0
\end{equation}
for all $t\geq t_0$, and
equation~\eqref{eq:apx_VS_conservation} gives
\begin{equation}
0\leq V_S(t)\leq V_0.
\end{equation}

Since $\mathcal D_{V_0}$ is compact and positively invariant, solutions
starting in $\mathcal D_{V_0}$ are defined for all $t\geq t_0$.

In the absence of vertical transmission, the local reproduction index of the
corresponding frozen model with standard linear infection pressure is, from
Appendix~\ref{apx:r(t)},
\begin{equation}
\mathcal R_l(t_0)
=
\sqrt{
H_{S,0}V_0
\frac{\beta_{hv}\beta_{vh}}
{\gamma\mu_0}
},
\qquad
H_{S,0}:=H_S(t_0).
\label{eq:apx_Rref}
\end{equation}

Throughout this supplementary section, $\mathcal R_l(t_0)$ refers to the index of the
\emph{corresponding linear frozen model}. It is not a local reproduction
threshold for the shrinking-adjusted system.

Indeed, at a disease-free state with susceptible host density
$\overline H_S$, the linearised infected subsystem of
equation~\eqref{eq:apx_frozen_shrinking} is
\begin{equation}
\begin{pmatrix}
H_I\\
V_I
\end{pmatrix}'
=
\begin{pmatrix}
-\gamma
&
\beta_{vh}\overline H_S
\\
0
&
-\mu_0
\end{pmatrix}
\begin{pmatrix}
H_I\\
V_I
\end{pmatrix},
\end{equation}
because
\begin{equation}
g'(0)=0.
\end{equation}
Its eigenvalues are $-\gamma$ and $-\mu_0$. Thus the disease-free set is
locally attracting in the infected directions for all parameter values.
Epidemic growth under the shrinking adjustment is therefore a
finite-amplitude nonlinear phenomenon.

\subsection*{A sufficient condition for extinction}

We first show that the classical threshold of the corresponding linear model
remains a sufficient condition for extinction under the shrinking adjustment.

\begin{proposition}
\label{prop:extinction}
Assume that
\begin{equation}
\mathcal R_l(t_0)<1.
\label{eq:apx_extinction_condition}
\end{equation}
Then, for every solution of
equation~\eqref{eq:apx_frozen_shrinking} starting in
$\mathcal D_{V_0}$, the infectious compartments converge exponentially to
zero:
\begin{equation}
H_I(t)\to0,
\qquad
V_I(t)\to0
\qquad
\text{as }t\to+\infty.
\end{equation}
Consequently, the corresponding full state converges exponentially fast to a disease-free state
of the form
\begin{equation}
\left(
H_{S,\infty},
0,
H_{R,\infty},
V_0,
0
\right).
\end{equation}
\end{proposition}

\begin{proof}

By construction,
\begin{equation}
0\leq g(H_I)\leq H_I
\qquad
\text{for all }H_I\geq0.
\label{eq:apx_g_upper}
\end{equation}

If $H_{S,0}=0$, then
\begin{equation}
H_S(t)=0
\qquad
\text{for all }t\geq t_0
\end{equation}
and
\begin{equation}
H_I(t)
=
H_I(t_0)e^{-\gamma(t-t_0)}.
\end{equation}
Moreover,
\begin{equation}
V_I'(t)
\leq
\beta_{hv}V_0H_I(t)
-\mu_0V_I(t),
\end{equation}
so a standard comparison argument shows that $V_I$ also converges
exponentially to zero. We may therefore assume $H_{S,0}>0$.

Condition~\eqref{eq:apx_extinction_condition} is equivalent to
\begin{equation}
\frac{\beta_{hv}V_0}{\gamma}
<
\frac{\mu_0}{\beta_{vh}H_{S,0}}.
\end{equation}
Hence there exists $\eta>0$ such that
\begin{equation}
\frac{\beta_{hv}V_0}{\gamma}
<
\eta
<
\frac{\mu_0}{\beta_{vh}H_{S,0}}.
\label{eq:apx_eta}
\end{equation}

Consider
\begin{equation}
L(t)
=
\eta H_I(t)+V_I(t).
\label{eq:apx_Lyapunov}
\end{equation}
Along solutions,
\begin{align}
L'(t)
&=
\eta
\left(
\beta_{vh}V_I(t)H_S(t)-\gamma H_I(t)
\right)
\nonumber\\
&\quad
+
\beta_{hv}g(H_I(t))
\left(
V_0-V_I(t)
\right)
-
\mu_0V_I(t).
\end{align}

Using
\begin{equation}
g(H_I)\leq H_I
\end{equation}
and
\begin{equation}
V_0-V_I(t)\leq V_0,
\end{equation}
we obtain
\begin{align}
L'(t)
&\leq
-
\left(
\eta\gamma-\beta_{hv}V_0
\right)H_I(t)
\nonumber\\
&\quad
-
\left(
\mu_0-\eta\beta_{vh}H_S(t)
\right)V_I(t).
\end{align}

Since
\begin{equation}
H_S'(t)
=
-\beta_{vh}V_I(t)H_S(t)
\leq0,
\end{equation}
we have
\begin{equation}
H_S(t)\leq H_{S,0}.
\end{equation}
Therefore,
\begin{equation}
L'(t)
\leq
-AH_I(t)-BV_I(t),
\end{equation}
where
\begin{equation}
A
:=
\eta\gamma-\beta_{hv}V_0
>0,
\qquad
B
:=
\mu_0-\eta\beta_{vh}H_{S,0}
>0.
\end{equation}

Define
\begin{equation}
\kappa
:=
\min
\left\{
\frac{A}{\eta},
B
\right\}
>0.
\end{equation}
Then
\begin{equation}
AH_I(t)+BV_I(t)
\geq
\kappa
\left(
\eta H_I(t)+V_I(t)
\right)
=
\kappa L(t),
\end{equation}
and hence
\begin{equation}
L'(t)
\leq
-\kappa L(t).
\end{equation}

By Gronwall's inequality,
\begin{equation}
L(t)
\leq
L(t_0)
e^{-\kappa(t-t_0)}
\qquad
\text{for all }t\geq t_0.
\end{equation}
Since
\begin{equation}
0\leq V_I(t)\leq L(t),
\qquad
0\leq H_I(t)\leq\frac{L(t)}{\eta},
\end{equation}
both infectious compartments converge exponentially to zero.

Finally, $H_S$ is non-increasing and bounded below by zero, whereas $H_R$
is non-decreasing and bounded above by $H_0$. Both therefore converge, with
\begin{equation}
H_{S,\infty}+H_{R,\infty}=H_0.
\end{equation}
Moreover,
\begin{equation}
V_I(t)\to0,
\qquad
V_S(t)\to V_0.
\end{equation}
\end{proof}

\begin{remark}
Proposition~\ref{prop:extinction} establishes global extinction of the
infectious compartments within the biologically admissible state space.
It does not imply convergence towards one prescribed disease-free
equilibrium, because the final partition of the host population between
susceptible and recovered individuals depends on the epidemic trajectory.
\end{remark}

\subsection*{A sufficient condition for finite-amplitude epidemic growth}

We next derive a sufficient condition ensuring that an introduction
corresponding to one infectious host can produce subsequent positive growth
of the infectious host compartment.

Define
\begin{equation}
a_0
:=
\mu_0+\beta_{hv}H_\star,
\qquad
\rho_0
:=
\frac{a_0}{2\gamma}.
\label{eq:apx_a_rho}
\end{equation}
Let
\begin{equation}
\tau_\star
=
\begin{cases}
\dfrac{\ln\rho_0}
{2\gamma(\rho_0-1)},
& \text{if }\rho_0\neq1,
\\[10pt]
\dfrac{1}{2\gamma},
& \text{if }\rho_0=1,
\end{cases}
\label{eq:apx_tau_star}
\end{equation}
and define
\begin{equation}
f:(0,+\infty)\to(1,+\infty),
\qquad
f(x)
=
\begin{cases}
x^{\frac{1}{x-1}},
& \text{if }x\neq1,
\\[6pt]
e,
& \text{if }x=1.
\end{cases}
\label{eq:apx_f}
\end{equation}

\begin{proposition}
\label{prop:emergence}
Assume
\begin{equation}
H_I(t_0)=H_\star,
\qquad
0\leq V_I(t_0)\leq V_0.
\end{equation}
If
\begin{equation}
V_0
\left(
H_{S,0}-\gamma H_\star\tau_\star
\right)
>
\frac{\gamma a_0}
{\beta_{hv}\beta_{vh}}
f(\rho_0),
\label{eq:apx_emergence_condition}
\end{equation}
then there exists
\begin{equation}
t^\star
\in
(t_0,t_0+\tau_\star]
\end{equation}
such that
\begin{equation}
H_I'(t^\star)>0.
\end{equation}
Thus the infectious host compartment necessarily undergoes a phase of
positive growth after introduction.
\end{proposition}

\begin{proof}

Assume by contradiction that
\begin{equation}
H_I'(t)\leq0
\qquad
\text{for all }
t\in[t_0,t_0+\tau_\star].
\end{equation}
Since
\begin{equation}
H_I(t_0)=H_\star,
\end{equation}
this implies
\begin{equation}
0\leq H_I(t)\leq H_\star
\qquad
\text{for all }
t\in[t_0,t_0+\tau_\star].
\label{eq:apx_HI_upper}
\end{equation}

Moreover,
\begin{equation}
H_I'(t)
\geq
-\gamma H_I(t),
\end{equation}
and therefore
\begin{equation}
H_I(t)
\geq
H_\star
e^{-\gamma(t-t_0)}.
\label{eq:apx_HI_lower}
\end{equation}

Define
\begin{equation}
q(t)
:=
\frac{H_I(t)^2}{H_\star}.
\end{equation}
Then
\begin{equation}
H_\star
e^{-2\gamma(t-t_0)}
\leq
q(t)
\leq
H_\star.
\label{eq:apx_q_bounds}
\end{equation}

Since the two branches of $g$ coincide at $H_I=H_\star$, the infectious
vector equation on this interval is
\begin{align}
V_I'(t)
&=
\beta_{hv}q(t)
\left(
V_0-V_I(t)
\right)
-\mu_0V_I(t)
\nonumber\\
&=
\beta_{hv}V_0q(t)
-
\left(
\mu_0+\beta_{hv}q(t)
\right)V_I(t).
\end{align}

Variation of constants gives
\begin{align}
V_I(t)
&=
V_I(t_0)
\exp
\left(
-\int_{t_0}^{t}
\left(
\mu_0+\beta_{hv}q(u)
\right)\,du
\right)
\nonumber\\
&\quad
+
\beta_{hv}V_0
\int_{t_0}^{t}
q(s)
\exp
\left(
-\int_s^t
\left(
\mu_0+\beta_{hv}q(u)
\right)\,du
\right)
\,ds.
\end{align}

Since
\begin{equation}
q(u)\leq H_\star,
\end{equation}
we have
\begin{equation}
\mu_0+\beta_{hv}q(u)
\leq
a_0.
\end{equation}
Using also the lower bound in
equation~\eqref{eq:apx_q_bounds} and dropping the non-negative initial term,
\begin{equation}
V_I(t)
\geq
\beta_{hv}V_0H_\star
\int_{t_0}^{t}
e^{-2\gamma(s-t_0)}
e^{-a_0(t-s)}
\,ds.
\end{equation}

For $\tau\geq0$, define
\begin{equation}
\varphi(\tau)
:=
\int_0^\tau
e^{-2\gamma s}
e^{-a_0(\tau-s)}
\,ds.
\end{equation}
Then
\begin{equation}
V_I(t)
\geq
\beta_{hv}V_0H_\star
\varphi(t-t_0).
\end{equation}

If $a_0\neq2\gamma$,
\begin{equation}
\varphi(\tau)
=
\frac{
e^{-2\gamma\tau}
-
e^{-a_0\tau}
}{
a_0-2\gamma
},
\end{equation}
whereas, if $a_0=2\gamma$,
\begin{equation}
\varphi(\tau)
=
\tau e^{-2\gamma\tau}.
\end{equation}
In both cases, $\varphi$ reaches its unique maximum at
$\tau=\tau_\star$, and
\begin{equation}
\varphi(\tau_\star)
=
\frac{1}{a_0f(\rho_0)}.
\label{eq:apx_phi_max}
\end{equation}

We now derive a lower bound for $H_S$. Since
\begin{equation}
(H_S+H_I)'
=
-\gamma H_I
\geq
-\gamma H_\star,
\end{equation}
we obtain
\begin{equation}
H_S(t)+H_I(t)
\geq
H_{S,0}+H_\star
-\gamma H_\star(t-t_0).
\end{equation}
Using $H_I(t)\leq H_\star$,
\begin{equation}
H_S(t)
\geq
H_{S,0}
-\gamma H_\star(t-t_0).
\label{eq:apx_HS_lower}
\end{equation}

Condition~\eqref{eq:apx_emergence_condition} implies
\begin{equation}
H_{S,0}
-\gamma H_\star\tau_\star
>0.
\end{equation}
Hence, at $t=t_0+\tau_\star$,
\begin{align}
H_I'(t_0+\tau_\star)
&=
\beta_{vh}
V_I(t_0+\tau_\star)
H_S(t_0+\tau_\star)
-
\gamma H_I(t_0+\tau_\star)
\nonumber\\
&\geq
\beta_{hv}\beta_{vh}
V_0H_\star
\left(
H_{S,0}
-\gamma H_\star\tau_\star
\right)
\frac{1}{a_0f(\rho_0)}
-
\gamma H_\star
\nonumber\\
&=
H_\star
\left[
\frac{
\beta_{hv}\beta_{vh}V_0
\left(
H_{S,0}
-\gamma H_\star\tau_\star
\right)
}{
a_0f(\rho_0)
}
-\gamma
\right].
\end{align}

Condition~\eqref{eq:apx_emergence_condition} therefore implies
\begin{equation}
H_I'(t_0+\tau_\star)>0,
\end{equation}
which is a contradiction.
\end{proof}

\begin{remark}
Condition~\eqref{eq:apx_emergence_condition} is sufficient but not
necessary. It is conservative because the proof replaces the time-varying
coefficient
\begin{equation}
\mu_0+\beta_{hv}q(t)
\end{equation}
by its upper bound
\begin{equation}
a_0=\mu_0+\beta_{hv}H_\star.
\end{equation}
\end{remark}

\begin{remark}
The growth condition is non-vacuous. For fixed
$V_0$, $\gamma$, $\mu_0$, $\beta_{hv}$, $\beta_{vh}$, and $H_\star$,
equation~\eqref{eq:apx_emergence_condition} is equivalent to
\begin{equation}
H_{S,0}
>
\gamma H_\star\tau_\star
+
\frac{
\gamma a_0
}{
\beta_{hv}\beta_{vh}V_0
}
f(\rho_0).
\end{equation}
The right-hand side is finite and independent of $H_{S,0}$.
Consequently, the sufficient condition holds for sufficiently large
susceptible-host density.
\end{remark}

\begin{remark}
Using equation~\eqref{eq:apx_Rref}, the sufficient growth condition can be
written as
\begin{equation}
\mathcal R_l(t_0)^2
\left(
1-
\frac{
\gamma H_\star\tau_\star
}{
H_{S,0}
}
\right)
>
\frac{a_0}{\mu_0}
f(\rho_0).
\label{eq:apx_emergence_Rref}
\end{equation}
Since
\begin{equation}
0<
1-
\frac{
\gamma H_\star\tau_\star
}{
H_{S,0}
}
<1,
\qquad
\frac{a_0}{\mu_0}>1,
\qquad
f(\rho_0)>1,
\end{equation}
this condition necessarily implies
\begin{equation}
\mathcal R_l(t_0)>1.
\end{equation}
This comparison does not make $\mathcal R_l$ a local threshold for the
shrinking-adjusted system: $\mathcal R_l$ remains the reproduction index of
the corresponding linear frozen model.
\end{remark}

\begin{figure}[h!]
    \begin{subfigure}{.49\textwidth}
        \centering
        \includegraphics[width=.98\linewidth]{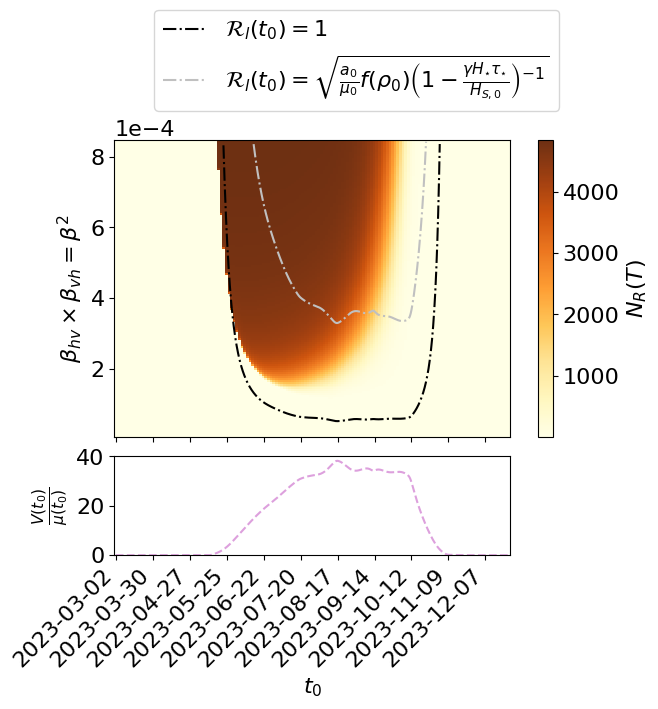}
        \caption{}
        \label{figapx:heatmap}
    \end{subfigure}\hfill
    \begin{subfigure}{.49\textwidth}
        \centering
        \includegraphics[width=.98\linewidth]{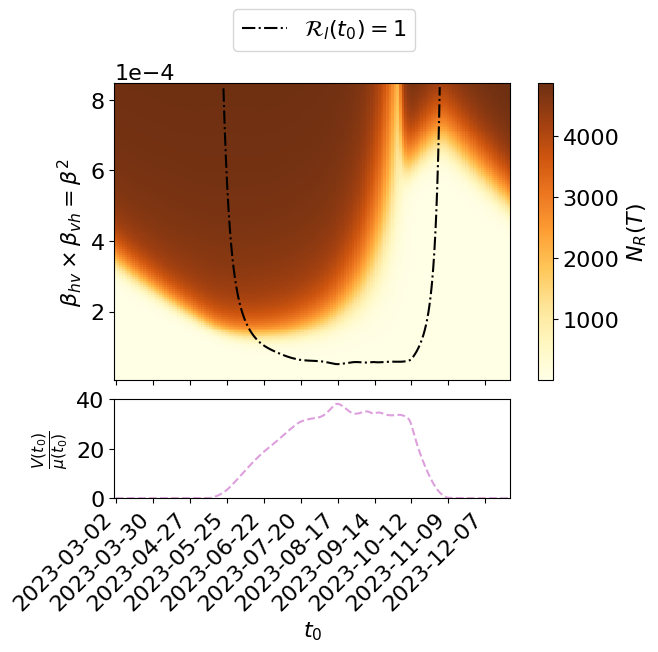}
        \caption{}
        \label{figapx:heatmap_with_sa}
    \end{subfigure}
\caption{\textit{
\textbf{Final number of recovered hosts $N_R(T)$} as a function of the
transmission rates and of the date of introduction $t_0$ (a) with shrinking adjustment, and (b) without 
shrinking adjustment (linear infection pressure).\\Prescribed vector abundance derived from entomological
surveillance data of \Aa\ collected in Decize, France, in 2023 and
2024. The corresponding evolution of the $V/\mu$ ratio is displayed below.\\The dashed curves represent the frozen-environment analytical criteria
obtained from Propositions~\ref{prop:extinction} and~\ref{prop:emergence},
evaluated at each introduction time $t_0$. For the simulations shown here,
$\gamma=1/5$ days$^{-1}$ and $\beta_{vh}=\beta_{hv}$.}}
\end{figure}

Figure~\ref{figapx:heatmap} illustrates the relationship between the
frozen-environment analytical criteria and numerical epidemic outcomes
obtained under time-varying prescribed vector abundance. For the full
time-varying simulations, these curves should therefore be interpreted as
reference criteria rather than as rigorous sufficient conditions. Their
agreement with numerical outcomes depends on the subsequent temporal
variation of the epidemiological environment.

Figure~\ref{figapx:heatmap_with_sa} emphasizes the importance of introducing shrinking adjustment for the local reproduction index to stay meaningful.


\clearpage

\section{Inter-seasonal extinction under time-varying vector abundance}
\label{apx:proof_winter_extinction}

In Section~\ref{apx:emgc_ext_under_shrinking}, extinction and epidemic
growth were analysed in an epidemiological environment frozen at the
introduction time. We now return to the full non-autonomous model, in which
the prescribed vector abundance $V(t)$ and mortality rate $\mu(t)$
vary with time.

We show that, under the shrinking adjustment, a sufficiently long period of
sufficiently low vector abundance can drive the infectious compartments into
a region from which they subsequently converge to zero, even if vector
abundance increases afterwards. This provides a deterministic mechanism for
suppressing artificial inter-seasonal re-emergence in the absence of pathogen
reintroduction.

We consider
\begin{equation}
\left\{
\begin{aligned}
H_S'(t)
&=
-\beta_{vh}V_I(t)H_S(t),
\\
H_I'(t)
&=
\beta_{vh}V_I(t)H_S(t)
-\gamma H_I(t),
\\
H_R'(t)
&=
\gamma H_I(t),
\\
V_I'(t)
&=
\beta_{hv}g(H_I(t))
\left(
V(t)-V_I(t)
\right)
-
\mu(t)V_I(t),
\end{aligned}
\right.
\label{eq:apx_nonauto_shrinking}
\end{equation}
where $g$ is defined in equation~\eqref{eq:apx_shrinking}.

Assume
\begin{equation}
V\in C^1([0,+\infty)),
\qquad
\mu\in C([0,+\infty)),
\end{equation}
and suppose that there exist
$\overline V>0$ and $\underline\mu>0$ such that
\begin{equation}
0\leq V(t)\leq\overline V,
\qquad
\mu(t)\geq\underline\mu
\qquad
\text{for all }t\geq t_0.
\label{eq:apx_global_bounds}
\end{equation}

We further assume the compatibility condition
\begin{equation}
V'(t)+\mu(t)V(t)\geq0
\qquad
\text{for all }t\geq t_0.
\label{eq:apx_G_compatibility}
\end{equation}
By Section~\ref{apx:critera_A4_A5}, if
\begin{equation}
0\leq V_I(t_0)\leq V(t_0),
\end{equation}
then
\begin{equation}
0\leq V_I(t)\leq V(t)
\qquad
\text{for all }t\geq t_0.
\label{eq:apx_admissible_VI}
\end{equation}

Let $t_1\geq t_0$ satisfy
\begin{equation}
H_S(t_1)>0,
\qquad
H_I(t_1)>0.
\label{eq:apx_winter_initial_positive}
\end{equation}
Since
\begin{equation}
H_S'(t)
=
-\beta_{vh}V_I(t)H_S(t),
\end{equation}
we have
\begin{equation}
H_S(t)>0
\end{equation}
for every finite $t\geq t_1$.

Fix
\begin{equation}
\delta
\in
\left(
0,
\min\{\gamma,\underline\mu\}
\right)
\end{equation}
and define
\begin{equation}
c_\delta(t)
:=
\frac{
\gamma-\delta
}{
\beta_{vh}H_S(t)
}.
\label{eq:apx_cdelta}
\end{equation}
Differentiating gives
\begin{equation}
c_\delta'(t)
=
\frac{
(\gamma-\delta)V_I(t)
}{
H_S(t)
}
\geq0,
\label{eq:apx_cdelta_derivative}
\end{equation}
so $c_\delta$ is non-decreasing.

Define
\begin{equation}
h_\delta
:=
\min
\left\{
\frac{H_\star}{2},
\frac{
H_\star c_\delta(t_1)
(\underline\mu-\delta)
}{
2\beta_{hv}\overline V
}
\right\},
\label{eq:apx_hdelta}
\end{equation}
and
\begin{equation}
\tau_\delta
:=
\max
\left\{
0,
\frac{1}{\gamma}
\ln
\left(
\frac{
2H_I(t_1)
}{
h_\delta
}
\right)
\right\}.
\label{eq:apx_taudelta}
\end{equation}

For every $\tau\geq\tau_\delta$, define
\begin{equation}
\varepsilon_{\delta,\tau}
:=
\min
\left\{
\frac{
\gamma h_\delta
}{
2\beta_{vh}H_S(t_1)
},
\;
c_\delta(t_1)
H_I(t_1)
e^{-\gamma\tau}
\right\}.
\label{eq:apx_epsdelta}
\end{equation}
Under
equation~\eqref{eq:apx_winter_initial_positive},
\begin{equation}
\varepsilon_{\delta,\tau}>0.
\end{equation}

\begin{proposition}
\label{prop:interseasonal}

Let $\delta$, $h_\delta$, $\tau_\delta$, and
$\varepsilon_{\delta,\tau}$ be defined above. For every
$\tau\geq\tau_\delta$, assume
\begin{equation}
V(t)
<
\varepsilon_{\delta,\tau}
\qquad
\text{for all }
t\in[t_1,t_1+\tau].
\label{eq:apx_low_vector_period}
\end{equation}
Set
\begin{equation}
t_2:=t_1+\tau.
\end{equation}
Then
\begin{equation}
H_I'(t)
\leq
-\delta H_I(t)
\leq0
\qquad
\text{for all }t\geq t_2.
\label{eq:apx_HI_decay}
\end{equation}
In particular, both $H_I(t)$ and $V_I(t)$ converge exponentially to zero as
$t\to+\infty$.
\end{proposition}

\begin{proof}

We divide the proof into two steps. First, we show that the low-vector period
drives $(H_I,V_I)$ into a shrinking-dominated region. Second, we show that
this region remains positively invariant after the low-vector period has
ended.

\medskip

\noindent
\textbf{Step 1. Entry into a shrinking-dominated region.}

For $t\geq t_1$, define
\begin{equation}
\Delta_\delta(t)
:=
\left\{
(h,v)\in\mathbb R_+^2:
0\leq h\leq h_\delta,
\;
0\leq v\leq c_\delta(t)h
\right\}.
\label{eq:apx_triangle}
\end{equation}

From equation~\eqref{eq:apx_admissible_VI},
\begin{equation}
0\leq V_I(t)\leq V(t).
\end{equation}
Hence, under
equation~\eqref{eq:apx_low_vector_period},
\begin{equation}
V_I(t)
<
\varepsilon_{\delta,\tau}
\qquad
\text{for all }t\in[t_1,t_2].
\end{equation}

Since $H_S$ is non-increasing,
\begin{align}
H_I'(t)
&=
\beta_{vh}V_I(t)H_S(t)
-
\gamma H_I(t)
\nonumber\\
&\leq
\beta_{vh}
\varepsilon_{\delta,\tau}
H_S(t_1)
-
\gamma H_I(t)
\end{align}
for $t\in[t_1,t_2]$.

Comparison with the corresponding linear equation gives
\begin{align}
H_I(t_2)
&\leq
e^{-\gamma\tau}
H_I(t_1)
\nonumber\\
&\quad
+
\frac{
\beta_{vh}
\varepsilon_{\delta,\tau}
H_S(t_1)
}{
\gamma
}
\left(
1-e^{-\gamma\tau}
\right).
\end{align}

By definition of $\tau_\delta$,
\begin{equation}
e^{-\gamma\tau}H_I(t_1)
\leq
\frac{h_\delta}{2},
\end{equation}
and by definition of $\varepsilon_{\delta,\tau}$,
\begin{equation}
\frac{
\beta_{vh}
\varepsilon_{\delta,\tau}
H_S(t_1)
}{
\gamma
}
\leq
\frac{h_\delta}{2}.
\end{equation}
Therefore
\begin{align}
H_I(t_2)
&\leq
\frac{h_\delta}{2}
+
\frac{h_\delta}{2}
\left(
1-e^{-\gamma\tau}
\right)
\nonumber\\
&=
h_\delta
\left(
1-\frac12e^{-\gamma\tau}
\right)
<h_\delta.
\label{eq:apx_HI_small_t2}
\end{align}

On the other hand,
\begin{equation}
H_I'(t)
\geq
-\gamma H_I(t),
\end{equation}
and hence
\begin{equation}
H_I(t_2)
\geq
H_I(t_1)e^{-\gamma\tau}
>0.
\label{eq:apx_HI_lower_t2}
\end{equation}

Using the definition of
$\varepsilon_{\delta,\tau}$ and the monotonicity of $c_\delta$,
\begin{align}
V_I(t_2)
&<
\varepsilon_{\delta,\tau}
\nonumber\\
&\leq
c_\delta(t_1)
H_I(t_1)
e^{-\gamma\tau}
\nonumber\\
&\leq
c_\delta(t_2)H_I(t_2).
\end{align}

Thus
\begin{equation}
0<H_I(t_2)<h_\delta,
\qquad
0\leq V_I(t_2)
<
c_\delta(t_2)H_I(t_2),
\label{eq:apx_entry_triangle}
\end{equation}
and consequently
\begin{equation}
(H_I(t_2),V_I(t_2))
\in
\Delta_\delta(t_2).
\end{equation}

\medskip

\noindent
\textbf{Step 2. Positive invariance of $\Delta_\delta(t)$.}

Since
\begin{equation}
h_\delta
\leq
\frac{H_\star}{2}
<
H_\star,
\end{equation}
the shrinking branch applies throughout $\Delta_\delta(t)$:
\begin{equation}
g(H_I)
=
\frac{H_I^2}{H_\star}.
\end{equation}

Suppose, by contradiction, that the solution leaves
$\Delta_\delta(t)$ after $t_2$. Define
\begin{equation}
t^\sharp
:=
\sup
\left\{
s\geq t_2:
(H_I(t),V_I(t))
\in
\Delta_\delta(t)
\text{ for every }t\in[t_2,s]
\right\},
\end{equation}
and assume
\begin{equation}
t^\sharp<+\infty.
\end{equation}

Since
\begin{equation}
H_I'(t)\geq-\gamma H_I(t)
\end{equation}
and $H_I(t_2)>0$, we have
\begin{equation}
H_I(t)>0
\end{equation}
for every finite $t\geq t_2$.

At $V_I=0$,
\begin{equation}
V_I'
=
\beta_{hv}g(H_I)V(t)
\geq0,
\end{equation}
so a first exit cannot occur through $V_I=0$. It also cannot occur through
$H_I=0$. Therefore, a first exit can occur only through
\begin{equation}
H_I=h_\delta
\end{equation}
or
\begin{equation}
V_I=c_\delta(t)H_I.
\end{equation}

\medskip

\noindent
\textit{Case 1.}
Suppose
\begin{equation}
H_I(t^\sharp)=h_\delta.
\end{equation}
At a first contact with this boundary,
\begin{equation}
H_I'(t^\sharp)\geq0.
\end{equation}
However,
\begin{align}
H_I'(t^\sharp)
&=
\beta_{vh}
V_I(t^\sharp)
H_S(t^\sharp)
-
\gamma h_\delta
\nonumber\\
&\leq
\beta_{vh}
c_\delta(t^\sharp)
h_\delta
H_S(t^\sharp)
-
\gamma h_\delta
\nonumber\\
&=
-\delta h_\delta
<0,
\end{align}
which is a contradiction.

\medskip

\noindent
\textit{Case 2.}
Suppose
\begin{equation}
V_I(t^\sharp)
=
c_\delta(t^\sharp)H_I(t^\sharp).
\end{equation}
Define
\begin{equation}
Z(t)
:=
V_I(t)-c_\delta(t)H_I(t).
\end{equation}
Since $Z(t)\leq0$ before the first exit and
$Z(t^\sharp)=0$,
\begin{equation}
Z'(t^\sharp)\geq0.
\label{eq:apx_Z_exit}
\end{equation}

Since
\begin{equation}
c_\delta'(t)\geq0,
\end{equation}
we have
\begin{align}
Z'(t^\sharp)
&=
V_I'(t^\sharp)
-
c_\delta(t^\sharp)H_I'(t^\sharp)
-
c_\delta'(t^\sharp)H_I(t^\sharp)
\nonumber\\
&\leq
V_I'(t^\sharp)
-
c_\delta(t^\sharp)H_I'(t^\sharp).
\end{align}

For readability, all quantities below are evaluated at $t^\sharp$. Using
\begin{equation}
V_I=c_\delta H_I,
\qquad
0\leq V_I\leq V\leq\overline V,
\qquad
\mu\geq\underline\mu,
\end{equation}
we obtain
\begin{align}
Z'
&\leq
\beta_{hv}
\frac{H_I^2}{H_\star}
(V-V_I)
-
\mu V_I
-
c_\delta
\left(
\beta_{vh}V_IH_S-\gamma H_I
\right)
\nonumber\\
&\leq
\beta_{hv}\overline V
\frac{H_I^2}{H_\star}
-
\underline\mu c_\delta H_I
-
\beta_{vh}c_\delta^2H_IH_S
+
\gamma c_\delta H_I.
\end{align}

By definition of $c_\delta$,
\begin{equation}
\beta_{vh}c_\delta(t)H_S(t)
=
\gamma-\delta,
\end{equation}
and hence
\begin{equation}
Z'(t^\sharp)
\leq
H_I(t^\sharp)
\left[
\beta_{hv}\overline V
\frac{
H_I(t^\sharp)
}{
H_\star
}
-
c_\delta(t^\sharp)
(\underline\mu-\delta)
\right].
\end{equation}

Since
\begin{equation}
H_I(t^\sharp)\leq h_\delta
\end{equation}
and $c_\delta$ is non-decreasing,
\begin{align}
Z'(t^\sharp)
&\leq
H_I(t^\sharp)
\left[
\beta_{hv}\overline V
\frac{h_\delta}{H_\star}
-
c_\delta(t_1)
(\underline\mu-\delta)
\right].
\end{align}
By definition of $h_\delta$,
\begin{equation}
\beta_{hv}\overline V
\frac{h_\delta}{H_\star}
\leq
\frac12
c_\delta(t_1)
(\underline\mu-\delta),
\end{equation}
so
\begin{equation}
Z'(t^\sharp)
\leq
-\frac12
c_\delta(t_1)
(\underline\mu-\delta)
H_I(t^\sharp)
<0.
\end{equation}
This contradicts
equation~\eqref{eq:apx_Z_exit}.

Both possible first-exit mechanisms lead to a contradiction. Therefore,
\begin{equation}
(H_I(t),V_I(t))
\in
\Delta_\delta(t)
\qquad
\text{for all }t\geq t_2.
\label{eq:apx_triangle_invariant}
\end{equation}

In particular,
\begin{align}
H_I'(t)
&=
\beta_{vh}V_I(t)H_S(t)
-\gamma H_I(t)
\nonumber\\
&\leq
\beta_{vh}
c_\delta(t)
H_I(t)H_S(t)
-\gamma H_I(t)
\nonumber\\
&=
-\delta H_I(t).
\end{align}
Hence
\begin{equation}
H_I(t)
\leq
H_I(t_2)
e^{-\delta(t-t_2)}
\leq
h_\delta
e^{-\delta(t-t_2)}.
\label{eq:apx_HI_exponential}
\end{equation}

Finally,
\begin{align}
V_I'(t)
&=
\beta_{hv}
\frac{H_I(t)^2}{H_\star}
\left(
V(t)-V_I(t)
\right)
-
\mu(t)V_I(t)
\nonumber\\
&\leq
\beta_{hv}\overline V
\frac{h_\delta^2}{H_\star}
e^{-2\delta(t-t_2)}
-
\underline\mu V_I(t).
\end{align}
Therefore,
\begin{align}
V_I(t)
&\leq
V_I(t_2)
e^{-\underline\mu(t-t_2)}
\nonumber\\
&\quad
+
\frac{
\beta_{hv}\overline Vh_\delta^2
}{
H_\star
}
\int_{t_2}^{t}
e^{-\underline\mu(t-s)}
e^{-2\delta(s-t_2)}
\,ds.
\label{eq:apx_VI_exponential_bound}
\end{align}

Writing
\begin{equation}
x=t-t_2,
\end{equation}
the convolution term is
\begin{equation}
\int_{t_2}^{t}
e^{-\underline\mu(t-s)}
e^{-2\delta(s-t_2)}
\,ds
=
\begin{cases}
\dfrac{
e^{-2\delta x}
-
e^{-\underline\mu x}
}{
\underline\mu-2\delta
},
& \text{if }\underline\mu\neq2\delta,
\\[10pt]
x e^{-\underline\mu x},
& \text{if }\underline\mu=2\delta.
\end{cases}
\end{equation}
In either case, there exist constants $C>0$ and $\kappa>0$ such that
\begin{equation}
V_I(t)
\leq
C e^{-\kappa(t-t_2)}
\qquad
\text{for all }t\geq t_2.
\end{equation}
For example, any
\begin{equation}
0<\kappa<
\min\{
\underline\mu,
2\delta
\}
\end{equation}
is admissible after adjusting $C$.

Thus both infectious compartments converge exponentially to zero.
\end{proof}

\begin{remark}
If
\begin{equation}
H_I(t_1)=V_I(t_1)=0,
\end{equation}
the system is already disease-free and remains so in the absence of pathogen
reintroduction. The assumption $H_I(t_1)>0$ in
Proposition~\ref{prop:interseasonal} therefore concerns the non-trivial case
in which a residual infectious host population remains at the beginning of
the low-vector period.
\end{remark}

\begin{remark}
The conclusion of Proposition~\ref{prop:interseasonal} is stronger than
extinction during the low-vector period itself. Once $(H_I,V_I)$ has entered
the invariant region $\Delta_\delta(t)$, the infectious compartments continue
to decay even if the prescribed vector abundance subsequently increases,
provided the global bound in
equation~\eqref{eq:apx_global_bounds} remains satisfied.
\end{remark}

\begin{remark}
The threshold $\varepsilon_{\delta,\tau}$ is a sufficient and deliberately
conservative bound. In particular, it depends on $\tau$ through the factor
$e^{-\gamma\tau}$. Proposition~\ref{prop:interseasonal} should therefore be
interpreted as stating that, for every sufficiently long low-vector period,
there exists a sufficiently small vector-abundance threshold ensuring entry
into the invariant extinction region. It does not imply that increasing
$\tau$ necessarily relaxes the required upper bound on $V(t)$.
\end{remark}

\newpage

\section*{Supplementary figures}

\begin{figure}[htb!]
    \begin{subfigure}{.49\textwidth}
        \centering
        \includegraphics[width=0.99\linewidth]{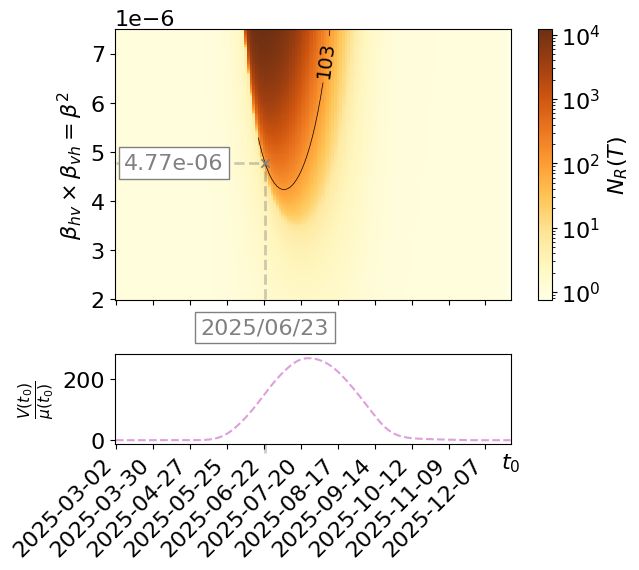}
        \caption{}
        \label{supfig:heatmap_Bergerac_2025}
    \end{subfigure}\hfill
    \begin{subfigure}{.49\textwidth}
        \centering
        \includegraphics[width=0.99\linewidth]{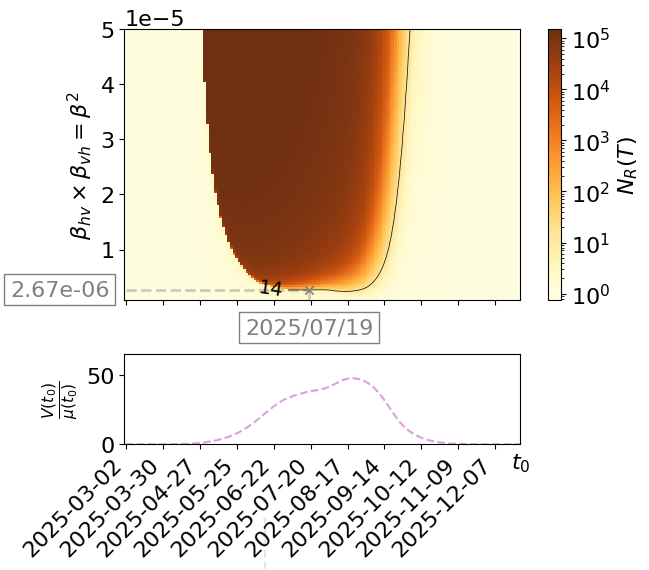}
        \caption{}
        \label{supfig:heatmap_Dijon_2025}
    \end{subfigure}
    \caption{\small\textit{\textbf{Total number of infected over the epidemic period in the host population} $N_{R}(T)$ simulated in terms of the transmission rates ($\beta_{hv}$ and $\beta_{vh}$) and of the date of introduction ($t_0$), with shrinking adjustment, and \textbf{including observed indigenous cases of chikungunya virus disease}. The corresponding evolution of the $V/\mu$ ratio is displayed below. $\gamma = 1/5$ days$^{-1}$, $\beta_{vh}=\beta_{hv}=: \beta$.
    \\\textbf{(a) Bergerac, France}: based on entomological surveillance data of eggs of \Aa\ in Bergerac, France, in 2025.  The continuous black line corresponds to the isocline $N_R(T) = 103$ (total number of reported cases during the Bergerac 2025 outbreak).
    \\\textbf{(b) Dijon, France}: based on entomological surveillance data of eggs of \Aa\ in Dijon, France, in 2025.
    The continuous black line corresponds to the isocline $N_R(T) = 14$ (total number of reported cases during the Dijon 2025 outbreak).}}
    \label{supfig:heatmaps_2025}
\end{figure}

\vspace{-1cm}

\begin{figure}[h]
     \centering
     \begin{subfigure}[b]{0.32\textwidth}
         \centering
         \includegraphics[width=\textwidth]{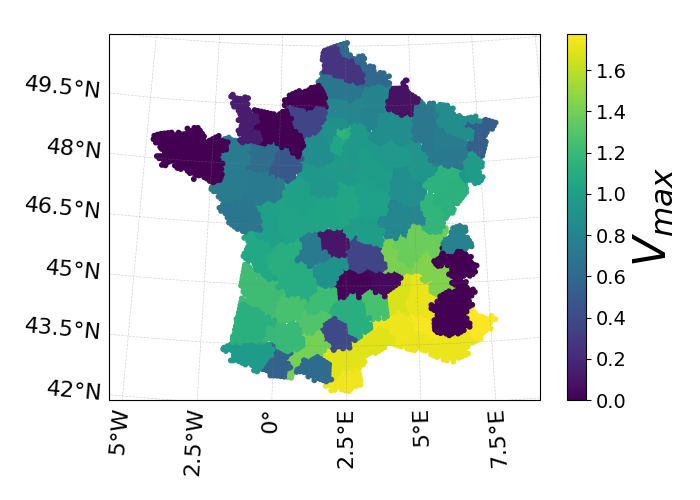}
         \caption{}
     \end{subfigure}
     \hfill
     \begin{subfigure}[b]{0.32\textwidth}
         \centering
         \includegraphics[width=\textwidth]{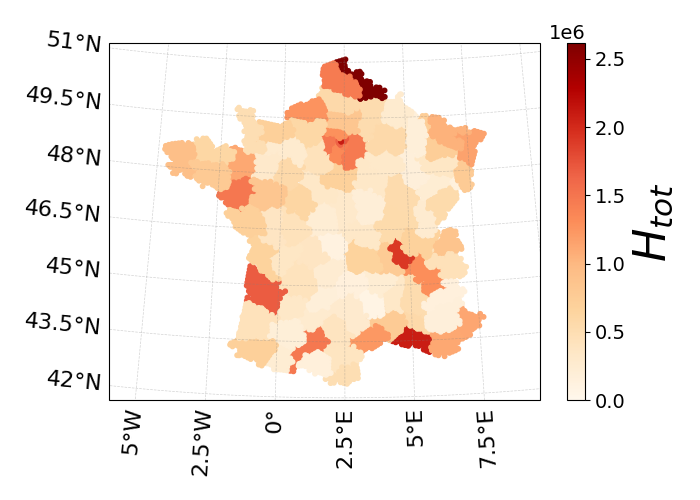}
         \caption{}
     \end{subfigure}
     \hfill
     \begin{subfigure}[b]{0.32\textwidth}
         \centering
         \includegraphics[width=\textwidth]{DISTIGRI_NRtot.png}
         \caption{}
     \end{subfigure}
        \caption{\textit{
        (a) : \textbf{Maximum number of vectors} per DISTIGRI trajectory per French continental department between July 24$^{th}$, 2024 ($t_0$) and November 2$^{nd}$, 2024. Initial condition : ECDC data in Sept. 2022 \cite{ECDC_Aedes_map}.\\
        (b) : \textbf{Population size} in each French continental department \cite{insee_recensement_2023}.
        \\(c) \textbf{Total number of infected hosts over the epidemic period $N_R(T)$} in each French continental department, with shrinking adjustment, based on DISTIGRI outputs, with introduction of one infected host per department at $t_0:$ July 24$^{th}$, 2024 with $\gamma = 1/5$, $\beta_{hv}=0.02$, $\beta_{vh}=0.008$}}.
        \label{supfig:maps}
\end{figure}

\end{document}